\documentclass[11pt]{amsart}
\usepackage[english]{babel}
\usepackage{amsmath}
\usepackage{amsfonts}
\usepackage{amssymb}
\usepackage{amsthm}
\usepackage{tikz}
\usetikzlibrary{calc}
\usepgflibrary{fpu}

\usepackage[english]{babel}
\usepackage{yfonts}
\usepackage[T1]{fontenc}
\usepackage{enumerate}
\usepackage{verbatim}
\usepackage{graphicx}
\usepackage{wrapfig}
\usepackage{verbatim}
\usepackage{faktor}
\usepackage{xcolor}
\usepackage[all]{xy}

\renewcommand{\>}{\rangle}

\newcommand{\Sp}{{\rm Sp}}
\newcommand{\Hom}{{\rm Hom}}
\renewcommand{\Im}{\mathrm{Im}}

\newcommand{\SU}{{\rm SU}}

\newcommand{\SO}{{\rm SO}}

\newcommand{\Id}{{\rm Id}}

\newcommand{\Hh}{\mathcal H}

\newcommand{\ov}{\overline}

\renewcommand{\b}{\beta}

\renewcommand{\k}{\kappa}

\renewcommand{\o}{\omega}
\newcommand{\G}{\Gamma}

\newcommand{\C}{\mathbb C}
\newcommand{\D}{\mathbb D}
\newcommand{\Dd}{\mathcal D}
\renewcommand{\H}{\mathbb H}

\newcommand{\R}{\mathbb R}

\newcommand{\Z}{\mathbb Z}

\newcommand{\calH}{\mathcal H}
\newcommand{\calX}{\mathcal X}
\newcommand{\calC}{\mathcal C}
\newcommand{\Xx}{\mathcal X}

\renewcommand{\frak}{\mathfrak}
\newcommand{\frl}{\mathfrak l}
\newcommand{\frgl}{\mathfrak {gl}}
\newcommand{\frsu}{\mathfrak {su}}
\newcommand{\frg}{\mathfrak g}
\newcommand{\frsp}{\mathfrak {sp}}
\newcommand{\frsl}{\mathfrak {sl}}
\newcommand{\frh}{\mathfrak {h}}

\newcommand{\fre}{\mathfrak{e}}
\newcommand{\frso}{\mathfrak{so}}

\newcommand{\rk}{{\rm rk}}
\newcommand{\rad}{{\rm rad}}

\newcommand{\Yy}{\mathcal Y}

\usepackage{stmaryrd}
\numberwithin{equation}{section}

\numberwithin{equation}{section}

\newcommand{\bqn}{\begin{equation*}}
\newcommand{\eqn}{\end{equation*}}

\newcommand{\bq}{\begin{equation}}
\newcommand{\eq}{\end{equation}}
\newcommand{\ba}{\begin{aligned}}
\newcommand{\ea}{\end{aligned}}
\newcommand{\be}{\begin{enumerate}}
\newcommand{\ee}{\end{enumerate}}
\newcommand{\bsm}{\left[\begin{smallmatrix}}
\newcommand{\esm}{\end{smallmatrix}\right]}                   
\newcommand{\bpm}{\begin{bmatrix}}
\newcommand{\epm}{\end{bmatrix}}
\newcommand{\barr}{\begin{displaymath}\begin{array}{cccc}}
\newcommand{\earr}{\end{array}\end{displaymath}}
\newcommand{\barrl}{\begin{displaymath}\begin{array}{lcl}}
\newcommand{\earrl}{\end{array}\end{displaymath}}
\newcommand{\barl}{\begin{displaymath}\begin{array}{l}}
\newcommand{\earl}{\end{array}\end{displaymath}}
\newcommand{\bxym}{ \begin{displaymath}\xymatrix }
\newcommand{\exym}{\end{displaymath}}

\theoremstyle{plain}
\newtheorem{thm}{Theorem}[section]
\newtheorem{lem}[thm]{Lemma}
\newtheorem{prop}[thm]{Proposition}
\newtheorem{cor}[thm]{Corollary}
\newtheorem*{teo*}{Theorem}

\theoremstyle{definition}
\newtheorem{defn}[thm]{Definition}
\newtheorem{ex}[thm]{Example}

\newcommand{\thismonth}{\ifcase\month 
  \or January\or February\or March\or April\or May\or June%
  \or July\or August\or September\or October\or November%
  \or December\fi}

\begin{document}
\title{Classification of tight homomorphisms }
\author{O. Hamlet}
\address{Department of Mathematics, Chalmers University of Technology
and the University of Gothenburg, 412 96 G\"OTEBORG, SWEDEN}
\email{hamlet@chalmers.se}
\author{M. B. Pozzetti}
\address{Department Mathematik, ETH Z\"urich, 
R\"amistrasse 101, CH-8092 Z\"urich, Switzerland}
\email{beatrice.pozzetti@math.ethz.ch}


\date{\today}
\thanks{O.H. author warmly acknowledges the hospitality of FIM institute in Zurich where this work was completed, B.P. was partially supported by the Swiss National Science Foundation project 200020-144373. }

\begin{abstract}
A totally geodesic map $f:\calX_1\to\calX_2$ between Hermitian symmetric spaces is tight if its image contains geodesic triangles of maximal area. Tight maps were first introduced in \cite{tight}, and were classified in \cite{Ham1,Ham2,Ham3} in the case of irreducible domain. We complete the classification by analyzing also maps from reducible domains. This has applications to the study of maximal representations and of the Hitchin component for representations in the groups $\Sp(2n,\R)$.
\end{abstract}
\maketitle

\section{Introduction}\label{sec:intro}
When $\Xx$ is a Hermitian symmetric space it is possible to use the K\"ahler form $\o$ of $\Xx$ to define an invariant of triples of points  $(x_0,x_1,x_2)$ in $\Xx$: the area, measured according to $\o$, of any totally geodesic triangle having $x_0,x_1,x_2$ as vertices. This invariant, that will be denoted by $\beta_\Xx$ in what follows, was studied by Domic and Toledo \cite{DT} in the case of classical domains and by Clerc and \O rsted \cite{CO} in the general case, who show that, denoting by $\rk(\Xx)$ the real rank of the symmetric space, it holds $|\beta_\Xx|<\pi\rk(\Xx)$
and the latter inequality is sharp. Here and in the following the Riemannian metric on a Hermitian symmetric space will always be normalized in such a way that its minimal holomorphic sectional curvature is equal to -1.

If now $\Xx_1$ and $\Xx_2$ are two Hermitian symmetric spaces, it becomes an interesting problem to determine if there exists a totally geodesic map $f:\Xx_1\to\Xx_2$ with the property that 
$$\sup_{(x_0,x_1,x_2)\in \Im(f)^3}|\beta_{\Xx_2}(x_0,x_1,x_2)|=\sup_{(y_0,y_1,y_2)\in \Xx_2^3}|\beta_{\Xx_2}(y_0,y_1,y_2)|$$
or if there is some geometric obstruction that excludes that the image of $f$ encompasses triangles of all the possible two dimensional area.

The class of maps for which the equality holds, the so called \emph{tight} maps, was first singled out by Burger, Iozzi and Wienhard in \cite{tight}, where a number of interesting geometric properties of those maps were established: we mention here that the image of $\Xx_1$ is a symmetric space with compact centralizer, and $f$ extends to a continuous map between the Shilov boundaries of the symmetric spaces. If we denote by $\Dd$ the Harish Chandra realization of a Hermitian symmetric space $\Xx$ as a bounded domain in a complex affine space, and by $G$ its isometry group, the Shilov boundary of $\Xx$ is the unique closed $G$-orbit in the topological boundary of $\Dd$.

The form $\o$ only measures area in the complex directions. Therefore if $f:\Xx\to\Yy$ is a totally real subspace then the function $\beta_\Yy$ vanishes on the image of $f$. More generally, given the close relation of tight maps with properties of the K\"ahler form, one might expect that tight maps need to be holomorphic. This is not always the case: it was shown in \cite{tight} that the totally geodesic map equivariant with the irreducible representation of $\frsl(2,\R)$ in $\frsp(2n,\R)$ is tight but not holomorphic. However the first author proved in his thesis \cite{Ham2,Ham3} that this is the only simple exception: if $\Xx$ is irreducible and not the Poincar\'e disc, and $f:\Xx\to\Yy$ is a tight map, then it is holomorphic. The first main result of the paper generalizes this statement to reducible domains:
\begin{thm}\label{thm:1}
 Let $\frg, \frh$ be Hermitian Lie algebras, assume that no simple factor of $\frg$ is isomorphic to $\frsu(1,1)$ and $\rho:\frg\to\frh$ is a homomorphism whose associated totally geodesic map $f$ is tight and positive. Then $f$ is holomorphic.
\end{thm}

In their study of tight homomorphisms between Hermitian Lie groups, Burger, Iozzi and Wienhard introduced the Hermitian hull of a subspace $\Xx$ of the Hermitian symmetric space $\Yy$. This is the smallest holomorphically embedded Hermitian symmetric subspace $\calH(\Xx)$ of $\Yy$ containing $\Xx$. Similarly, whenever $\rho:\frg\to\frh$ is a homomorphism, we will denote by $\calH(\rho)$ the smallest holomorphically embedded Hermitian subalgebra of $\frh$ containing $\rho(\frg)$. If $\Xx$ is the symmetric space associated to $\frg$, and $f$ is the totally geodesic map that is $\rho$-equivariant, the Lie algebra $\calH(\rho)$ corresponds, under the natural correspondence between subalgebras and symmetric subspaces, to $\calH(f(\Xx))$. Our second result is the explicit description the Hermitian hull of a tight map (a Lie algebra homomorphism will from now on be said to be tight, positive, holomorphic, if the corresponding totally geodesic map is):
\begin{thm}\label{thm:1.2}\label{thm:2}
 Let $\frg=\frsu(1,1)^l\oplus \frg_1\oplus\ldots\frg_k$ be a Hermitian Lie algebra, with $\frg_i$ not isomorphic to $\frsu(1,1)$ and of noncompact type. Let $\rho:\frg\to \frh$ be an injective, positive, tight homomorphism, and assume further that $\frh$ has no simple factor isomorphic to $\fre_{7(-25)}$. Then 
 \begin{enumerate}
  \item  the Hermitian hull $\calH(\rho)$ is of the form
 $$\mathcal H(\rho)=\bigoplus_{j=1}^l\bigoplus_{l=1}^{n_j}\frsp(2m_{jl},\R)\oplus \frg_1\oplus\ldots\oplus\frg_k,$$
\item the inclusion $\calH(\rho)\to \frh$ is holomorphic and tight,
\item the homomorphism $\rho:\frg\to\calH(\rho)$ is a direct sum of irreducible representations on the $\frsu(1,1)$ factors and identifications on the factors $\frg_i$.
 \end{enumerate}
\end{thm}
In \cite{Ham1} the first author described all possible tight holomorphic homomorphisms between Hermitian Lie algebras. Combining that description with the result of Theorem \ref{thm:2} we obtain, in the last section, complete lists of all the tight embeddings between Hermitian Lie algebras: 
\begin{thm}\label{thm:3}
 The lists in Section \ref{sec:lists} provide an exhaustive list of all tight embeddings between Hermitian Lie algebras
 with the possible exception of some homomorphisms with values in $\fre_{7(-25)}$.
\end{thm}
The main difference between the lists in our last section and the ones in \cite{Ham1} is that, instead of referring to root subsystems, we give an elementary description of all the homomorphisms involved, that can be directly used in applications.

The necessity of removing from the main theorem of \cite{Ham2} the irreducibility hypothesis on the domain is not just an academic exercise. One of the most notable applications of tight maps is to the study of maximal representations. These are interesting and well studied lattice homomorphisms.
To be more precise, if $\G$ denotes 
a lattice in a Hermitian Lie group $G$, $H$ another Hermitian Lie group, a homomorphism $\rho:\G\to H$ is a \emph{maximal representation} if it satisfies some cohomological hypothesis that should be thought of as a generalization of a property characterizing homomorphisms associated to points in the Teichm\"uller component \cite{BIW,BGPG}. It is shown in \cite{BIW} that if $\rho:\G\to H$ is a maximal representation, then the Zariski closure $L$ of $\rho(\G)$ is a reductive Hermitian Lie group whose associated symmetric space $\Yy$ is tightly embedded in the symmetric space associated to $H$. However, in general $\Yy$ is not irreducible, it is therefore important to understand tight homomorphisms from reducible domains.   

As a corollary of our work we get restrictions on the subgroups of $H$ that can arise as Zariski closure of the image $\rho(\G)$ of a maximal representation. This can, in turn, be used to deduce properties of maximal representations: for example the description of tightly embedded subalgebras of $\frsu(m,n)$ carried out in the last section of this paper is used in \cite{Poz} to deduce rigidity results for maximal representations of complex hyperbolic lattices in $\SU(m,n)$.

When $G$ is $\Sp(2m,\R)$, that is the only family of Hermitian Lie groups that are also split Lie groups, the Hitchin component of $\Hom(\G_g,G)$ consists only of maximal representations. This implies that if $\rho:\G_g\to \Sp(2m,\R)$ is a representation in the Hitchin component, and $H={\ov{\rho(\G)}}\phantom!^Z$ is the Zariski closure of the image, then the Lie algebra $\frh$ belongs to the list of possible tightly embedded subalgebras of $\frsp(2m,\R)$ determined in Section \ref{sec:lists}. 
Homomorphisms $\rho:\G_g\to\Sp(4,\R)$ that belong to the Hitchin component parametrize properly convex foliated projective contact structures on the unit tangent bundle of a surface \cite{GW}. In this case our result imposes strong restrictions on the possible Zariski closure of the image of $\rho$ (that can only be $\SU(1,1)$, $\SU(1,1)^2$ or $\Sp(4,\R)$). This restricts the possible geometric structures compatible with convex foliated real projective structures.

\subsection*{Structure of the paper}
In Section \ref{sec:Preliminaries} we recall the results we will need about linear representations of Lie algebras, complex and quaternionic linear algebra and tight homomorphisms.
In Section \ref{sec:branching} we describe some branching arguments: if the target is either $\frsu(m,n)$ or $\frso^*(2p)$, then one can use some orthogonal decomposition of $\C^{m+n}$ (resp. $\H^p$) into irreducible modules to split a tight homomorphism into a direct sum of easier homomorphisms that can be well understood.
In Section \ref{sec:proofs} we deduce Theorems \ref{thm:1} and \ref{thm:2} from the branching argument of Section \ref{sec:branching}. In Section \ref{sec:lists} we combine Theorem \ref{thm:1} and \ref{thm:2} with the classification of tight holomorphic homomorphisms in \cite{Ham1} to give an explicit description of all tight homomorphisms with codomain different from $\fre_{7(-25)}$.  
\section{Preliminaries}\label{sec:Preliminaries}

\subsection{(Linear) representations of Lie algebras}
Let $\rho:\frg\to\frsu(m,n)$ be a homomorphism, and let us consider the associated linear representation $\ov \rho$ of $\frg$ on $\C^{m+n}$ that is obtained by composing $\rho$ with the standard inclusion $\frsu(m,n)\to \frgl(\C^{m+n})$. A subspace $V$ of $\C^{m+n}$ is $\ov\rho(\frg)$\emph{-invariant} if $\ov\rho(X)v\in V$ for each $v$ in $V$ and $X$ in $\frg$. The linear representation $\ov\rho$ is \emph{irreducible} if the only $\ov\rho(\frg)$-invariant subspaces are $\{0\}$ and $\C^{m+n}$.

If $\frg$ is Hermitian, it is in particular a reductive Lie algebra, and hence it follows from Weyl's theorem \cite[Theorem 6.3]{Humphreys} that every representation of $\frg$ can be decomposed, uniquely up to conjugation in $\frgl(\C^{m+n})$, as a direct sum of irreducible representations. This means that there exists a $\ov\rho(\frg)$-invariant decomposition $\C^{m+n}=V_1\oplus\ldots \oplus V_k$ and homomorphisms $\rho_i:\frg\to\frgl(V_i)$ such that  $\ov\rho=i\circ \oplus \rho_i$ where $i:\oplus \frgl(V_i)\to \frgl(\C^{m+n})$ is the inclusion induced by the splitting.


It is well known that it is possible to classify irreducible linear representations of simple Lie algebras in terms of the associated weights. However, it is important to keep in mind that two different homomorphism $\rho_i:\frg\to\frsu(m,n)$ can be conjugate in $\frgl(\C^{m+n})$ without being conjugate in $\frsu(m,n)$, and even worse, the notion of tightness is not well defined for conjugacy classes of linear representations (see \cite[Section 4.3]{Ham2} for more details).

In the case of reducible Lie algebras $\frl=\frl_1\oplus\frl_2$, if $\rho^i:\frl_i\to \frgl(V_i)$ are representations, the tensor product representation $\rho^1\boxtimes \rho^2:\frl_1\oplus\frl_2\to \frgl(V_1\otimes V_2)$ 
is defined by the expression 
$$\rho^1\boxtimes \rho^2(l_1,l_2)(v_1\otimes v_2)=\rho^1(l_1)v_1\otimes v_2+v_2\otimes \rho^2(l^2)v_2.$$
It is well known  that every irreducible linear representation $\ov\rho:\frl_1\oplus\frl_2\to \frgl(\C^{m+n})$ can be uniquely expressed as $\rho^1\boxtimes \rho^2$ for some irreducible representations $\rho^i:\frl_i\to \frgl(V_i)$. Moreover we have:
\begin{lem}\label{lem:1}
Let $\rho_i:\frl_i\to\frgl(V_i)$ be representations. If the decomposition of $\rho_i$ in irreducible representations is $\rho_i=\oplus_{j=1}^{n_i}\rho_i^j$, then the decomposition of $\rho_1\boxtimes\rho_2$ in irreducible representations is 
$\bigoplus_{j=1}^{n_1}\bigoplus_{l=1}^{n_2}\rho_1^j\boxtimes\rho_2^l.$
\end{lem}
\begin{proof}
 This follows from the definition of the tensor product representation.
\end{proof}

Let us now denote by $h$  be the standard Hermitian form of signature $(m,n)$ on $\C^{m+n}$, that defines the algebra $\frsu(m,n)$.
A subspace $V$ of $\C^{m+n}$  is \emph{nondegenerate} if the restriction of of the form $h$ to $V$ has trivial radical: i.e. there is no vector that is orthogonal, with respect to $h$, to the whole space. 
Let us now fix a homomorphism $\rho:\frg\to \frsu(m,n)$. Assume that there exists a $\rho(\frg)$-invariant orthogonal decomposition $\C^{m,n}=V_1\oplus\ldots V_k$ where the restriction of $h$ to $V_i$ is nondegenerate and has signature $(m_i,n_i)$. In this case it is easy to check that the representation $\rho$ can be written as $\rho=\iota\circ\oplus\rho_i$ where $\rho_i:\frg\to\frsu(m_i,n_i)$ are homomorphisms and $\iota:\oplus\frsu(m_i,n_i)\to\frsu(m,n)$ is the inclusion associated to the splitting $\C^{m,n}=V_1\oplus\ldots V_k$.
However this is not always the case as exploited in the following example:
\begin{ex}
 Let us consider the homomorphism $\rho:\frgl(\C^2)\to\frsu(2,2)$ given by $X\mapsto \bsm X^a&X^h\\X^h&X^a\esm$ where $X^a=X^*-X$ is the anti-Hermitian part of the matrix $X$ and $X^h=X+X^*$ is the Hermitian part. The only $\rho(\frgl(\C^2))$-invariant subspaces of $\C^2$ are $\<e_1+e_3,e_2+e_4\>$ and $\<e_1-e_3, e_2-e_4\>$, but both subspaces are isotropic for the form $h$ defining $\frsu(2,2)$.
\end{ex}

We will say that a Lie algebra homomorphism $\rho:\frg\to \frsu(m,n)$ is \emph{irreducible} if the associated linear representation is.

We denote by $\H$ the division algebra of quaternions, that is the 4 dimensional real vector space with base $\{1,i,j,k\}$ together with the unique $\R$ bilinear multiplication satisfying $i^2=j^2=k^2=ijk=-1$, and we fix the identification of $\C^2$ with $\H$ by the right $\C$ linear map $(x,y)\mapsto x+jy$. 

If $V$ is a finite dimensional right module over the quaternions, we denote by $\frgl_\H(V)$ the algebra of (right) $\H$-linear endomorphisms of $V$. We moreover denote by $H$ a nondegenerate, quaternionic valued, antiHermitian form on $V$: this means that for each $a,b$ in $\H$ and each $v,w$ in $V$ we have 
$$H(va,wb)=\ov bH(v,w)a=-\ov{H(wb,va)}.$$
It is well known (cfr. \cite[Pages 277-278]{Satake}) that one of the realizations of the algebra $\frso^*(2p)$ is as the subalgebra of $\frgl_\H(\H^p)$ preserving the form $H$.

Let now $\rho:\frg\to\frso^*(2p)$ be a homomorphism. A quaternionic submodule $W$ of $\H^p$ is $\rho(\frg)$-invariant if $\rho(g)w\in W$ for all $w$ in $W$, and is \emph{irreducible as a quaternionic module} if the only $\rho(\frg)$-invariant quaternionic submodules of $W$ are $W$ and $\{0\}$. We will say that a homomorphism $\rho:\frg\to\frso^*(2p)$ is an \emph{irreducible quaternionic representation} if $\H^p$ is an irreducible quaternionic module.

The $\C$ linear identification of $\C^2$ with $\H$  has inverse the map $a+ib+jc+kd\mapsto(a+ib,c-id)$.
The standard inclusion $\tau:\frso^*(2p)\to\frsu(p,p)$ is obtained by identifying $\H^p$ with $\C^{2p}$ in the way explained above by forgetting the possibility of multiplying by $j$, and by realizing that the antiHermitian form $H$ can be written as $ih+jT$ where $h$ is a Hermitian form of signature $(p,p)$ and $T$ is an orthogonal form (cfr. \cite[Pages 277-278]{Satake}).

Let now $\rho:\frg\to \frso^*(2p)$ be an irreducible quaternionic representation. It is not always true that the composition $\tau\circ \rho:\frg\to\frsu(p,p)$ is irreducible: for example the composition of the tight embeddings $\frsu(p,p)\to\frso^*(4p)\to\frsu(2p,2p)$ gives a diagonal embedding, that is clearly not irreducible (cfr. Section \ref{sec:lists} for more details).
Similarly we have:
\begin{lem}\label{lem:2.4}
 Let $\rho:\frg\to\frso^*(2p)$ be a homomorphism. And let $V<\C^{2p}$ be an irreducible complex module. Then either $V$ is a quaternionic submodule of $\H^p$ or $V$ and $Vj$ are disjoint and anti-isomorphic.
\end{lem}
\begin{proof}
It follows from the formula we gave for the identification of $\H$ with $\C^2$ that the right multiplication by $j$ on $\H^{p}$ induces an antilinear map $J:\C^{2p}\to\C^{2p}$ that commutes with all the elements of $\tau\circ\rho(\frg)$ since $\rho(\frg)$ is contained in the algebra of quaternion linear endomorphisms of $\H^p$. In particular if $V$ is a vector subspace of $\C^{2p}$, then $JV$ is a vector subspace of $\C^{2p}$ that is anti-isomorphic to $V$. This means that if the restriction of $h$ to $V$ has signature $(m,n)$, then the restriction to $JV$ has signature $(n,m)$.
Moreover, whenever $V$ is an irreducible $\tau\circ\rho(\frg)$-module, we get that $V\cap JV$ is a $\tau\circ\rho(\frg)$-submodule that is hence equal either to $V$ or $\{0\}$ by irreducibility. If $V\cap JV=V$ we get that $V$ is a quaternionic module. Otherwise $V$ and $JV$ are disjoint and anti-isomorphic.
\end{proof}
\subsection{Tight homomorphisms}
Recall from the introduction that we denote by $\beta_\Xx$ the \emph{Bergmann cocycle}, that is defined as 
$$\beta_\Xx(x_0,x_1,x_2)=\int_{\Delta(x_0,x_1,x_2)}\o$$ 
where $\o$ denotes the K\"ahler form of the symmetric space $\Xx$ whose Riemannian metric is normalized so that it has holomorphic sectional curvature bounded below by -1, and the notation $\Delta(x_0,x_1,x_2)$ stands for any geodesic triangle with vertices $(x_0,x_1,x_2)$. It was proven in \cite{DT,CO} that the norm of the Bergmann cocycle satisfies $\|\b_\Xx\|=\pi\cdot\rk(\Xx)$.
\begin{defn}
A totally geodesic map $f:\Xx\to\Yy$ is \emph{tight} if and only if $\|f^*\beta_\Yy\|_\infty=\|\beta_\Yy\|_\infty$. 
\end{defn}

A fundamental tool in the study of tight homomorphisms is continuous bounded cohomology. Recall that if $G$ is a locally compact group the continuous bounded $n$-cochains of $G$ is the Banach module consisting of functions
$${\rm C}^n_{cb}(G,\R)=\{f:G^{n+1}\to \R|\; f \text{ is continuous and bounded}\}$$
endowed with the supremum norm and the $G$-action by diagonal multiplication.
The continuous bounded cohomology of $G$,  ${\rm H}^*_{cb}(G,\R)$, is the cohomology of the subcomplex $({\rm C}^n_{cb}(G,\R)^G,d)$ consisting of $G$-invariant functions. Continuous bounded cohomology of locally compact groups was introduced in \cite{BMJEMS} where a number of fundamental results were proven, including that, in degree 2, the norm on ${\rm C}_{cb}^2(G,\R)$ induces a Banach norm on ${\rm H}^2_{cb}(G,\R)$, that is normally known as the Gromov norm, and denoted by $\|\cdot\|_\infty$. Moreover, in degree 2, the following holds:
\begin{thm}[\cite{BMJEMS}]
 Let $G$ be a semisimple Lie group with finite center and no compact factor, and let $\Xx$ be the associated symmetric space. It holds 
 $$\Omega^2(\Xx,\R)^G={\rm H}_{cb}^2(G,\R).$$
\end{thm}
This says that the second bounded cohomology of a semisimple Lie group is generated by the K\"ahler forms of the irreducible factors of the associated symmetric space that are of Hermitian type. If $G$ is a Hermitian Lie group, we will denote by $\k^b_G$ the \emph{bounded K\"ahler class}: the class in ${\rm H}_{cb}^2(G,\R)$ that corresponds to the K\"ahler form of the symmetric space. We notice here that if $G$ is semisimple and can be decomposed as $G=G_1\times\ldots\times G_k$ then it holds $\k^b_G=\k^b_{G_1}+\ldots+\k^b_{G_k}$.  
The Gromov norm of the K\"ahler class was computed by Domic and Toledo for classical domains and by Clerc and \O rsted for general cases:
\begin{thm}[\cite{DT,CO,tight}]
 Let $G$ be a Hermitian Lie group of real rank $n$. Then
 $$\|\k^b_G\|_\infty=n\pi.$$
\end{thm}
Any totally geodesic map $f:\Xx\to\Yy$ between symmetric spaces corresponds to a homomorphism $\rho:G\to H$ between the associated isometry groups (possibly replacing $G$ by a covering group). 
A crucial observation of \cite{tight} is the following equivalent characterization of tightness for a totally geodesic map $f$ in term of metric properties of the pullback map in bounded cohomology associated to the homomorphism $\rho$:
\begin{prop}[{\cite[Proposition 6]{tight}}]\label{prop:bc}
 Let $\Xx,\Yy$ be Hermitian symmetric spaces with isometry group $G,H$. Let $f:\Xx\to\Yy$ be a totally geodesic map, and let $\rho:G\to H$ be the corresponding group homomorphism. Then $f$ is tight if and only if 
 $$\|\rho^*\k_H^b\|_\infty=\|\k^b_H\|_\infty.$$
\end{prop}
We will often use Proposition \ref{prop:bc} and switch between these two characterizations of tight maps. Moreover, using the correspondence between Lie algebra homomorphisms ${\rm d} \rho:\frak m\to \frak g$, totally geodesic maps $f_\rho:\Xx\to \Yy$, and Lie group homomorphisms $\rho:M\to G$,  we will say that a homomorphism $\rho$ is  \emph{holomorphic} if the associated totally geodesic map is holomorphic. Similarly we will say that a Lie algebra homomorphism ${\rm d} \rho:\frak m\to \frak g$ is \emph{tight} if its associated totally geodesic map, or equivalently, its group homomorphism, is.
A totally geodesic map $f:\Xx\to \Yy$  is \emph{positive} if for any triple $(x_0,x_1,x_2)\in \Xx^3$ with $\beta_\Xx(x_0,x_1,x_2)>0$ we have $\beta_\Yy(f(x_0),f(x_1),f(x_2))\geq0$, and \emph{strictly positive} if the strict inequality holds. This property has a bounded cohomological counterpart: let $\rho_f:G\to H$ be the group homomorphism associated to $f$. It is easy to check that  if the group $G$  splits as a product $G=G_1\times\ldots\times G_k$, and $\rho_f^*(\k_H^b)=\sum \alpha_i\k^b_{G_i}$ then the map $f$ is positive if  all $\alpha_i$ are greater than or equal to 0 and is strictly positive if all $\alpha_i$ are strictly positive.

An important geometric feature of any Hermitian symmetric space $\Xx$ is that the maximal flats can be complexified to obtain a holomorphic and isometric embedding of ${\rk(\Xx)}$ copies of the Poincar\'e disc, $\D^{\rk(\Xx)}\to \Xx$. These subspaces are called \emph{polydiscs} and are all conjugate under the action of the isometry group of $\Xx$. 
\begin{defn}\label{defn:2.7}
A \emph{diagonal disc} is the composition of the diagonal inclusion of the Poincar\'e disc in $\D^{\rk(\Xx)}$ and a maximal polydisc. 
\end{defn}
It is easy to check that a diagonal disc is an example of a tight map: indeed since the inclusion $f$ of a polydisc in $\Xx$ is holomorphic and isometric, one gets that $f^*\o_{\Xx}=\o_{\D^{\rk(\Xx)}}$ and moreover it is easy to check that, denoting by $\Delta$ the diagonal inclusion, one gets $\Delta^*\o_{\D^{\rk(\Xx)}}=\rk(\Xx)\o_{\D}$. This implies that: 
$$\sup_{(x_0,x_1,x_2)\in \Im(d)^3}|\beta_\Xx(x_0,x_1,x_2)|=\rk(\Xx)\pi,\, \sup_{(y_0,y_1,y_2)\in \D^3}|\beta_\D(y_0,y_1,y_2)|=\|\beta_\Xx\|_\infty\pi.$$

The following proposition summarizes easy properties of tight homomorphisms:
\begin{prop}\label{prop:tighteasy}
 \begin{enumerate}
  \item If a map $\rho:\frak m\to \frg_1\oplus\ldots\oplus \frg_n$ is tight  all the induced maps $\rho_i:\frak m\to \frg_i$ are tight.
  \item Let $\rho:\frak m\to \frg$ and $\tau:\frg\to \frl$ be homomorphisms. If the composition $\tau\circ \rho$ is tight and $\frg$ is simple, or $\tau$ is strictly positive, then both $\rho$ and $\tau$ are tight. If $\rho$ is tight and $\tau$ is tight and positive, then $\tau\circ \rho$ is tight.
 \end{enumerate}
\end{prop}
\begin{proof}
(1) is Lemma 3.7 in \cite{Ham2};
 (2) is a combination of Lemma 3.3 and Lemma 3.5 in \cite{Ham2}.
\end{proof}

We will also need the following lemma on the existence of diagonal discs:
\begin{lem}\label{lem:discs}
 Let $\rho:\frg_1\oplus\frg_2\to \frh$ be a tight homomorphism. Then there exist maps $d_i:\frsu(1,1)\to\frg_i$ such that, denoting by $d:\frsu(1,1)^2\to\frg_1\oplus\frg_2$ the direct sum of the homomorphisms $d_i$, we have $\rho \circ d:\frsu(1,1)^2\to\frh$ is tight.
\end{lem}
\begin{proof}
Let us denote by $G_1^1,\ldots, G_1^{k_1}$ and $G_2^1,\ldots,G_2^{k_2}$ the simple factors of $G_1$ and $G_2$ respectively. Let $\rho:G_1\times G_2\to H$ denote also the Lie group homomorphism associated to the Lie algebra homomorphism $\rho$. We have 
$$\rho^*\k_H^b=\sum_{j=1}^{k_1}\alpha_1^j\k^b_{G_1^j}+\sum_{j=1}^{k_2}\alpha_1^j\k^b_{G_2^j}$$
for some coefficients $\alpha_i^j$. Moreover, since $\rho$ is tight we have $\sum_{i,j} |\alpha_i^j|\cdot\rk(G^j_i)=\rk( H)$. Let us denote by $d_i^j:\frsu(1,1)\to \frg_i^j$ a diagonal disc, and by $\ov d_i^j$ an antidiagonal disc, that means the composition $d_i^j\circ a$ where $a:\frsu(1,1)\to\frsu(1,1)$ is the antiholomorphic isomorphism given by $\bsm a&b\\\ov b&\ov a\esm\mapsto \bsm\ov a&-\ov b\\- b& a\esm.$ It is easy to check, using the same argument as after Definition \ref{defn:2.7}, that $(d_i^j)^*\o_{G_i^j}=\rk (G^i_j)\o_{\SU(1,1)}$ and hence $(d_i^j)^*\k^b_{G_i^j}=\rk(G_i^j)\k^b_{\SU(1,1)}$. We define $$f_i^j=\begin{cases}
d_i^j \mbox{, if } \alpha_i^j\geq 0,\\
\ov d_i^j \mbox{, if } \alpha_i^j< 0.
\end{cases}$$
We thus get that the compositions $d_i:\frsu(1,1)\to\frg_i$ of a diagonal embedding $d:\frsu(1,1)\to\frsu(1,1)^{k_i}$ and the map $\oplus f_i^j:\oplus\frsu(1,1)\to \oplus\frg_i^j$ satisfy the assumption of the lemma.
\end{proof}

Recall that Hermitian Lie algebras are divided into two categories, according to whether the associated symmetric space is biholomorphic to a domain of the form $V+i\Omega$ where $V$ is a real vector space and $\Omega\subset V$ is a proper convex open cone. The algebras for which this holds are called of \emph{tube type} and are $\frsu(m,m),\frso^*(4m),\frsp(2m,\R), \frso(2,m),\fre_{7(-25)}$, the algebras in the second category, the ones that are \emph{not of tube type}, are $\frsu(m,n),\frso^*(4n+2),\fre_{6(-14)}$.
An important geometric feature of tight homomorphisms, discovered in \cite{tight}, is that they essentially do not mix between tube type algebras and algebras that are not tube type, as made precise in the next proposition.  
\begin{prop}[{\cite[Theorem 9]{tight}}]\label{prop:tt}
 Let $\rho:\frg_1\to \frg_2$ be a tight homomorphism, then
 \begin{enumerate}
  \item if $\frg_1$ is of tube type, then there exists a maximal subalgebra $\frg_2^T$ of $\frg_2$ of tube type such that $\rho(\frg_1)\subseteq \frg_2^T$, 
  \item if $\rho$ is injective and $\frg_2$ is of tube type, then $\frg_1$ is of tube type.
 \end{enumerate}
\end{prop}
 An irreducible homomorphism of $\frsu(1,1)$ in $\frsu(p,q)$ is tight if and only if its highest weight is odd: two different proofs of this fact can be found in \cite[Lemma 9.5]{tight} and \cite[Theorem 6.1]{Ham2}. By choosing, in each case, a suitable subalgebra isomorphic to $\frsu(1,1)$ it is possible to prove the following:
\begin{prop}[{\cite[Theorems 6.2 to 6.5]{Ham2}}]\label{prop:hamlowrk}
For the homomorphisms of the low rank algebras into $\frsu(m,n)$ we have:
\begin{enumerate}
  \item[2.] Each tight irreducible homomorphism of $\frsu(1,1)\oplus \frsu(1,1)$ factors through one of the two factors.
  \item[3.] The only tight irreducible homomorphism of $\frsp(4,\R)$ is the standard representation. In particular it is either holomorphic or antiholomorphic.
  \item[4.] Each tight irreducible homomorphism of $\frsp(4,\R)\oplus \frsu(1,1)$ factors through one of the factors. In particular, they are holomorphic in the $\frsp(4,\R)$-factor.
  \item[5.]The only tight irreducible homomorphisms of $\frsu(1,2)$ is the standard representation and its dual. In particular they are either holomorphic or antiholomorphic.
 \end{enumerate}
\end{prop}

As anticipated in the introduction, essentially all tight homomorphism except the ones defined on the Poincar\'e disc are holomorphic. An important tool to make this statement precise, that is also needed in the formulation of Theorem \ref{thm:1.2}, is the notion of Hermitian hull:
\begin{defn}
Let $\rho:\frg_1\rightarrow \frg_2$ be a homomorphism between Hermitian Lie algebras. The \emph{Hermitian hull}, $\mathcal{H}(\rho)=\mathcal{H}(\rho,\frg_1)$, is the smallest holomorphically embedded subalgebra  $\frg\subset\frg_2$ that contains $\rho(\frg_1)$.
\end{defn}
\begin{lem}\label{lem:hulhol}
Let $\eta\circ\rho:\frg_1\rightarrow\frg_2\rightarrow\frg_3$ be homomorphisms such that $\eta$ is holomorphic and injective. Then  $\mathcal{H}(\eta\circ\rho)$ is isomorphic to $\mathcal{H}(\rho)$.
\end{lem}
\begin{proof}
The subalgebra 
$\eta(\mathcal{H}(\rho))$, that is clearly isomorphic to $\mathcal{H}(\rho)$, is holomorphically embedded in $\frg_3$ and contains $\eta(\rho(\frg_1))$, thus we just need to show its minimality among all holomorphically embedded subalgebras. Suppose $\frg$ is a smaller subalgebra with this property, that means 
$$\eta(\rho(\frg_1))\subset\frg\subsetneq \eta(\mathcal{H}(\rho)).$$
We get, by the injectivity of $\eta$, that 
$$\rho(\frg_1)\subset\eta^{-1}(\frg)\subsetneq \mathcal{H}(\rho).$$
Moreover, since $\eta$ is holomorphic, $\eta^{-1}(\frg)$ is a holomorphically embedded subalgebra. This contradicts the definition of $\mathcal{H}(\rho)$, hence the result follows.
\end{proof}

\begin{lem}\label{lem:hul}
Let $\rho:\frg=\frg_1\oplus...\oplus\frg_n\rightarrow \frh_1\oplus...\oplus\frh_n=\frh$ be a homomorphism such that $\rho(\frg_i)\subset\frh_i$.
Then $\mathcal{H}(\rho,\frg)=\bigoplus_i\mathcal{H}(\rho|_{\frg_i},\frg_i)$.
\end{lem}
\begin{proof}
 Clearly the subalgebra $\bigoplus_i\mathcal{H}(\rho|_{\frg_i},\frg_i)$ is a holomorphically embedded subalgebra of $\frh$ containing $\rho(\frg)$, in particular $\mathcal{H}(\rho,\frg)<\bigoplus_i\mathcal{H}(\rho|_{\frg_i},\frg_i)$. Moreover it is easy to check that the intersection of two holomorphically embedded subalgebras is a holomorphically embedded subalgebra, hence $\calH(\rho,\frg)\cap \frh_i$ is a holomorphically embedded subalgebra of $\frh_i$ that contains $\rho|_{\frg_i}(\frg_i)$. This implies that, for each $i$, 
 $$\calH(\rho|_{\frg_i},\frg_i)<\calH(\rho,\frg)\cap \frh_i<\calH(\rho,\frg)$$
 and hence $\mathcal{H}(\rho,\frg)=\bigoplus_i\mathcal{H}(\rho|_{\frg_i},\frg_i)$.
\end{proof}

\section{Some branching arguments}\label{sec:branching}
The purpose of this section is to show how to exploit the structure of the algebras $\frsu(m,n)$ and $\frso^*(2p)$ to simplify the study of tight homomorphisms defined on reducible algebras. In Proposition \ref{prop:2.1} and \ref{prop:2.2} we show that  whenever a tight homomorphism $\rho:\frg\to\frh$ is fixed, where $\frh$ is either $\frsu(m,n)$ or $\frso^*(2p)$, 
it is possible to find a decomposition of the associated linear representation in irreducible modules that is compatible with the form defining the algebra $\frh$. Then, in Propositions \ref{prop:splitting2} and \ref{prop:splitting3}, we study such building blocks and use Proposition \ref{prop:hamlowrk} to show that they need to have a very simple form. 

\begin{prop}\label{prop:2.1}
 Let $\rho:\frg\to\frsu(m,n)$ be a tight homomorphism. Then there exists a decomposition $\C^{m+n}=V_1\oplus\ldots\oplus V_k$ consisting of nondegenerate irreducible, pairwise orthogonal subspaces. 
\end{prop}
\begin{proof}
We argue by induction on the number of irreducible subspaces $l$. The case $l=1$ is easily seen to be true.
Now assume $l>1$ and fix a decomposition $\C^{m+n}=V_1\oplus\ldots\oplus V_l$  into irreducible subspaces. 

Assume first that there exists an index $i$ such that the restriction of $h$ to $V_i$ is nondegenerate. Since the representation $\rho$ has values in $\frsu(m,n)$, the orthogonal complement of $V_i$ with respect to $h$, $V_i^\bot$, is a $\rho(\frg)$-module. The restriction of $h$ to $V_i^\bot$ is nondegenerate, otherwise the form $h$ itself would be degenerate. The $\rho(\frg)$-module $V_i^\bot$ admits, by induction on the number of irreducible subspaces in its decomposition, a $\rho(\frg)$-invariant orthogonal decomposition into irreducible $\rho(\frg)$-modules $V_i^\bot=V'_2\oplus\ldots\oplus V'_l$ where the restriction of $h$ to all $V_i$ is nondegenerate.  In turn $\C^{m+n}$ admits the desired splitting $V_i\oplus V'_2\oplus\ldots\oplus V'_l$, and we are done.

Next we assume that the restriction of $h$ to $V_i$ is degenerate for all $i$. We will argue that this contradicts the tightness assumption. If the restriction of $h$ to $V_i$ is degenerate it is easy to check that the restriction of $h$ to $V_i$ is identically zero: indeed the radical
 $$\rad(V_i)=\{v\in V_i|\; h(v,w)=0 ,\; \forall w\in V_i\}$$
is a nontrivial $\rho(\frg)$-invariant subspace that needs to equal the whole subspace by irreducibility.

 Our next goal is to show that, since $h$ is nondegenerate, there exists a module $V_j$ in the decomposition such that $h|_{V_1 \oplus V_j}$ is nondegenerate. 
Indeed, since $h$ is nondegenerate there exists a submodule $V_j$ such that $h|_{V_1 \oplus V_j}$ is nonzero, moreover for such a choice of $V_j$ we have $\rad(V_1 \oplus V_j)=\{0\}$: if, by contradiction, $a+b$ is an element in $\rad(V_1 \oplus V_j)$ with $a\in V_1$ and $b\in V_j$, one gets that $a$ is orthogonal to $V_j$ since $h(a,c)=h(a+b,c)=0$ for all $c\in V_j$. In particular, since $V_1$ is irreducible, one gets that $V_1$ and $V_j$ are orthogonal, and hence  the restriction of $h$ to $V_1 \oplus V_j$ vanishes.

Let us denote by $W$ the orthogonal complement of $V_1 \oplus V_j$, considered with respect to the form $h$. Since $(V_1 \oplus V_j)$ is $\rho(\frg)$-invariant, the subspace $W$ is as well. Moreover, if we denote by $k$ the dimension of $V_1$, it is easy to check that the signature of the restriction of $h$ to $V_1 \oplus V_j$ is $(k,k)$ and the restriction to $W$ has signature $(m-k,n-k)$.

Since the splitting is $\rho(\frg)$-invariant, we get that the image of $\rho$ is contained in $( \frgl(V_1)\oplus \frgl(V_j)\oplus \frgl(W))\cap \frsu(m,n)$. Our next aim is to show that this last group is equal to $\frgl(k)+\frsu(m-k,n-k)$. Since $h|_W$ has signature $(m-k,n-k)$ we get that $\frgl(W)\cap \frsu(m,n)=\frsu(m-k,n-k)$. We will now verify that $(\frgl(V_1)\oplus \frgl(V_j))\cap\frsu(k,k)$ is, in a suitable basis, the embedding of $\frgl(k)$ in $\frsu(k,k)$ defined by $X\mapsto \bsm X&\\&-X^*\esm$.
Let us fix any basis $\mathcal B_1$ of $V_1$. Since $V_1$ and $V_j$ are isotropic subspaces of $V_1 \oplus V_j$ and the restriction of $h$ to $V_1 \oplus V_j$ is nondegenerate, we can find a basis $\mathcal B_j$ of $V_j$ such that the matrix representing $h$ with respect to the basis $\mathcal B_1\cup \mathcal B_j$ is $\bsm 0&\Id\\\Id&0\esm$. When this basis is fixed it follows from the definition of the group $\frsu(k,k)$ that a pair $(X,Y)\in \frgl(V_1)\oplus\frgl(V_j)$ belongs to $\frsu(k,k)$ if and only if $Y=-X^*$.

However the inclusion $\iota=\iota_1+\iota_2:\frgl(k)\oplus\frsu(m-k,n-k)\to\frsu(m,n)$ cannot be tight: the group $\frgl(k)$ is not  Hermitian, hence $\iota_1^*\k^b_{\frsu(m,n)}=0$,  and the inclusion $\iota_2$ is holomorphic and isometric, thus $\iota_2^*\k^b_{\frsu(m,n)}=\k^b_{\frsu(m-k,n-k)}$. As $\rho$ factors via $\iota$ this implies that $\rho^*\k^b_{\frsu(m,n)}$ has norm at most $m-k$ which contradicts the tightness assumption.
\end{proof}
This readily implies the following:
\begin{cor}\label{cor:3.1}
 Any tight homomorphism $\rho:\frg\to\frsu(m,n)$ can be written as $\iota\circ\oplus\rho_i$ where $\rho_i:\frg\to \frsu(m_i,n_i)$ are tight homomorphisms whose associated representations are irreducible and $\iota$ corresponds to a holomorphic isometric embedding $\iota: \oplus \frsu(m_i,n_i)\to \frsu(m,n)$.
\end{cor}
We now turn to the study of tight homomorphisms with values in the group $\frso^*(2p)$. Using the same notation as in Section \ref{sec:Preliminaries}, we denote by $H$ the anti-Hermitian form on $\H^p$ with respect to which the group $\frso^*(2p)$ is defined, and we denote by $h$ on $\C^{2p}$ the Hermitian form that corresponds to the imaginary part of the form $H$.  It is well known that  the standard inclusion $\tau:\frso^*(2p)\to\frsu(p,p)$ corresponds to a holomorphic map and, given our normalization convention for the metrics, it is isometric up to a factor $2$. As in Lemma \ref{lem:2.4} we will denote by $J:\C^{2p}\to \C^{2p}$ the antiholomorphic isomorphism that corresponds to the right multiplication by $j$ on $\H^p$.
 
\begin{prop}\label{prop:2.2}
 Let $\rho:\frg\to\frso^*(2p)$ be a tight injective homomorphism. Then there exists a decomposition $\H^{p}=V_1\oplus\ldots\oplus V_l$ consisting of nondegenerate, pairwise $H$-orthogonal, $\rho(\frg)$-invariant subspaces, with the additional property that $V_i$ is either irreducible as complex representation or splits as the direct sum of two anti-isomorphic representations.
\end{prop}

\begin{proof}
 Let us first assume that $\frg$ is of tube type. As a consequence of Proposition \ref{prop:tt} we get that the image of $\rho$ is contained in a maximal tube type subdomain of $\frso^*(2p)$. It is well known that if $p$ is even $\frso^*(2p)$ is of tube type and if $p=2s+1$ is odd the maximal tube type subdomain of $\frso^*(4s+2)$  is $\frso^*(4s)$. This latter subalgebra of $\frso^*(4s+2)$ corresponds to a nondegenerate quaternionic submodule $V$ of $\H^{2s+1}$ of quaternionic dimension $2s$. In particular, by considering the $\rho(\frg)$-invariant nondegenerate orthogonal splitting $\H^p=V\oplus V^\bot$, we can restrict to the case when $p=2s$ is even.
 
Assuming $p=2s$ is even, we have that the inclusion $\tau:\frso^*(4s)\to \frsu(2s,2s)$ is tight and hence, as a consequence of Proposition \ref{prop:2.1}, we have that $\C^{4s}$ admits a decomposition into nondegenerate, irreducible $\tau\circ \rho(\frg)$-modules. Let $V$ be a complex module in such a decomposition and consider the span $V_1=\<V,JV\>$. The space $JV$ is an invariant module as well, this follows from that the representation $\tau\circ\rho$ is $\H$-linear. The space $V_1$ is thus a quaternionic submodule of $\H^{p}$, i.e. $\rho(\frg)$-invariant. Moreover the restriction of $H$ to $V_1$ is nondegenerate since the restriction of $h$ to $V$ was nondegenerate by assumption. By Lemma \ref{lem:2.4}, $V_1$ is either irreducible or splits as the direct sum of two anti-isomorphic modules.
 
Let $W$ denote the orthogonal complement of $V_1$ with respect to the form $H$. Since the representation $\rho$ has values in $\frso^*(2p)$ we get that $W$ is $\rho(\frg)$-invariant, and this concludes, by induction, the proof in the case $\frg$ is of tube type.
 
Next we consider the case where $\frg$ is not of tube type. From the assumption of tightness and Proposition \ref{prop:tt} it follows that $\frso^*(2p)$ is not of tube type either. The inclusion $\tau:\frso^*(2p)\to \frsu(p,p)$ is thus not tight. However, since $\tau$ corresponds to an holomorphic and, up to a factor 2, isometric map, we get that $\tau^*\k^b_{\SU(p,p)}=2\k^b_{\SO^*(2p)}$. 

We prove, again, the proposition by induction on the number of irreducible subspaces $l$ in the decomposition $\C^{2p}=V_1\oplus\ldots \oplus V_l$ (with respect to the representation $\tau\circ\rho$). If one of the modules $V_i$ is nondegenerate, we can apply the same trick as in the first part of the proof and construct a nondegenerate $\rho(\frg)$-invariant quaternionic module $V_i^\H$. In particular we can restrict to the orthogonal complement $(V_i^\H)^\bot$ and conclude the proof by induction.
 
 We are left to deal with the case in which each module in the decomposition $\C^{2p}=V_1\oplus\ldots\oplus V_l$ is degenerate with respect to the form $h$. In this case, following the lines of Proposition \ref{prop:2.1}, we get that we find two modules $V_1,V_i$ of complex dimension $k$ such that the restriction of $h$ to $V_1\oplus V_i$ is nondegenerate, and the same argument as in the proof of Proposition \ref{prop:2.1} implies that the representation $\tau\circ\rho$ can be written as $\iota\circ \rho_r$ where $\rho_r:\frg\to\frgl(k)\oplus \frsu(p-k, p-k)$ is a homomorphism and $\iota$ is an inclusion that is holomorphic and isometric in the $\frsu(p-k,p-k)$ factor.
 
If $k\geq 2$ this will contradict the assumption that $\rho$ is tight as:
\begin{align*} 
\|\rho^*2\k_{\SO^*(2p)}^b\|&=\|(\tau\circ\rho)^*\k^b_{\SU(p,p)}\|=\|(\iota\circ\rho_r)^*\k^b_{\SU(p,p)}\|\\
&\leq \|\iota^*\k^b_{\SU(p,p)}\|=\|\k^b_{\SU(p-k,p-k)}\|=p-k
\end{align*}
 
If instead $k=1$ we get that the representation $\pi_2\circ\rho_r:\frg\to\frsu(p-1,p-1)$ is tight since: 
\begin{align*}
\|(\pi_2\circ\rho_r)^* \k^b_{\SU(p-1,p-1)}\|&=\|\rho_r^* \iota^*\k^b_{\SU(p,p)}\| =\|(\tau\circ\rho)^*(\k^b_{\SU(p,p)})\|\\
&=\|\rho^*2(\k^b_{\SO^*(2p)})\|=\| 2\k^b_{\SO^*(2p)}\| =p-1.
\end{align*}
 But this contradicts Proposition \ref{prop:tt} as it would be a tight homomorphism from an algebra $\frg$ that is not of tube type to an algebra $\frsu(p-1,p-1)$ that is of tube type. This concludes the proof.
\end{proof}

Again an immediate application is the following:
\begin{cor}\label{cor:3.2}
Any tight homomorphism $\rho:\frg\to\frso^*(2p)$ can be written as $\iota\circ\oplus\rho_i$ where $\rho_i:\frg\to \frso^*(2p_i)$ are tight homomorphisms whose associated complex representations are either irreducible or direct sums of two anti-isomorphic representations and $\iota$ corresponds to a holomorphic isometric embedding.
\end{cor}
With Propositions \ref{prop:2.1} and \ref{prop:2.2} at our disposal we can now reduce tight homomorphisms $\rho:\frg_1\oplus\frg_2\to\frh$ to direct sums of homomorphisms defined on the factors:
\begin{prop}\label{prop:splitting2}
 Let $\rho:\frg_1\oplus\frg_2\to \frsu(m,n)$ be a tight homomorphism. Then there exists a subalgebra of $\frsu(m,n)$ of the form $\iota:\frsu(m_1,n_1)\oplus\frsu(m_2,n_2)\to\frsu(m,n)$ and tight homomorphisms $\rho_i:\frg_i\to\frsu(m_i,n_i)$ such that $\rho=\iota\circ(\rho_1\oplus\rho_2)$.
\end{prop}
\begin{proof}
 Using Corollary \ref{cor:3.1} we know that $\rho$ has values in a holomorphically and tightly embedded subalgebra of $\frsu(m,n)$ of the form $\frsu(m_1,n_1)\oplus\ldots\oplus\frsu(m_k,n_k)$, and splits as a direct sum of irreducible tight homomorphisms. Since the inclusion $\bigoplus\frsu(m_i,n_i)\to\frsu(m,n)$ is holomorphic, hence in particular strictly positive, we get, from the second fact in Proposition \ref{prop:tighteasy}, that the map $\rho:\frg_1\oplus\frg_2\to\bigoplus\frsu(m_i,n_i)$ is tight, and hence the composition $\rho_i$ of $\rho$ with the projection on $\frsu(m_i,n_i)$ is tight as well (as a consequence of the first fact of Proposition \ref{prop:tighteasy}). 
 
 In particular, in order to prove the proposition, it is enough to show that if $\rho:\frg_1\oplus\frg_2\to\frsu(m,n)$ is a tight homomorphism that is irreducible as a complex representation, then $\rho$ factors through one of the two factors. In turn this easily follows from the case 6.2 of Proposition \ref{prop:hamlowrk}: assume, by contradiction, that $\rho$ does not factor through one of the two factors, and let $d_i:\frsu(1,1)\to\frg_i$ be sums of diagonal or antidiagonal discs with the property that the composition $\rho|=\rho\circ(d_1\oplus d_2):\frsu(1,1)^2\to\frsu(m,n)$ is tight (cfr. Lemma \ref{lem:discs}). Using Proposition \ref{prop:2.1} and Lemma \ref{lem:1} we can write $\rho|$ as a sum $\oplus\rho_j$ of irreducible representations, of which at least one is non-tight. This gives a contradiction and concludes the proof. 
 \end{proof}
 
 A first useful Corollary of Proposition \ref{prop:splitting2} is the following, that complements the results of Proposition \ref{prop:hamlowrk}:
\begin{cor}\label{cor:ham}
 Let $\rho=\rho^{1}\boxtimes\rho^2:\frak{su}(1,2)\oplus\frak{su}(1,1)\to\frsu(m,n)$ be an irreducible homomorphism. If $\rho$ is tight, then either $\rho^1$ is trivial or $\rho^1$ is the standard representation and $\rho^2$ is trivial. This implies that every tight homomorphism $\rho:\frsu(1,2)\oplus\frsu(1,1)\to \frsu(m,n)$ is either holomorphic or antiholomorphic in the first factor.
\end{cor}
\begin{proof}
 It follows from Proposition \ref{prop:splitting2} that the homomorphism factors through one of the two factors. The fact that the only irreducible tight homomorphism of $\frsu(1,2)$ is associated to the standard representation is case 5 of Proposition \ref{prop:hamlowrk}.
\end{proof}
A proposition similar in spirit to Proposition \ref{prop:splitting2} for homomorphisms with values in $\frso^*(2p)$ is the following:
\begin{prop}\label{prop:splitting3}
 Let $\rho:\frg_1\oplus\frg_2\to \frso^*(2p)$ be a tight homomorphism. Then there exists a subalgebra of $\frso^*(2p)$ of the form $\iota:\frso^*(2p_1)\oplus\frso^*(2p_2)\to \frso^*(2p)$ and tight homomorphisms $\rho_i:\frg_i\to\frso^*(2p_i)$ such that $\rho=\iota\circ (\rho_1\oplus\rho_2)$.
\end{prop}
\begin{proof}
The same argument of the proof of Proposition \ref{prop:splitting2}, using Corollary \ref{cor:3.2} instead of Corollary \ref{cor:3.1},  implies that we can assume that 
$\rho:\frg\to\frso^*(2p)$ is tight and the composition $\tau\circ \rho:\frg\to\frsu(p,p)$ is either irreducible or splits as direct sum of two anti-isomorphic representations. We need to show that  in this case $\rho$ factors through one of the factors.

Let us first assume that $p$ is even. In this case the inclusion $\tau:\frso^*(2p)\to \frsu(p,p)$ is tight and holomorphic. If $\tau\circ\rho$ is irreducible, it follows directly from Proposition \ref{prop:splitting2} that $\tau\circ\rho$, and hence $\rho$ itself, factors through one of the two factors $\frg_i$. If instead $\tau\circ\rho$ is a direct sum of two anti-isomorphic representations, we deduce from Proposition \ref{prop:splitting2} that each of those representations factors through one of the two factors $\frg_i$, that needs to be the same for both since the representations are anti-isomorphic. This concludes the proof in the first case.  

Let us now assume that $p=2k+1$ is odd and consider the diagram
 $$\xymatrix{&t_1:\frg_1\oplus\frg_2\ar[r]^-\rho&\frso^*(4k+2)\ar[r]_-{\text{NT}}^-\tau&\frsu(2k+1,2k+1)\\
 &t_2:\frsu(1,1)\oplus\frsu(1,1)\ar[r]^-{\rho|}\ar[u]^d&\frso^*(4k)\ar[u]_{i_2}\ar[r]^{\tau|}&\frsu(2k,2k)\ar[u]^{\text{NT}}_{i_3}.}$$
 
 Here, again, the map $d$ is chosen in such a way that $\rho\circ d$ is tight, this is possible by Lemma \ref{lem:discs}. 
 The homomorphism $\rho|$ is then the composition of $\rho\circ d$, and its image  is contained in $\frso^*(4k)$, a maximal tube type subdomain of $\frso^*(4k+2)$, as a consequence of Proposition \ref{prop:tt}.
 
We get from the case 2 of Proposition \ref{prop:hamlowrk} that the decomposition of $\tau|\circ\rho|$ in irreducible representations contains only homomorphisms factoring through one of the two factors. Using Lemma \ref{lem:1} we get that this is only possible if each irreducible factor in the decomposition of the representation associated to the homomorphism $\tau\circ\rho$ itself was factoring through one of the two factors. Since, by hypothesis, there are at most two anti-isomorphic irreducible factors in the decomposition of $\tau\circ\rho$, we get that $\tau\circ\rho$ factors through one of the factors, and the same clearly holds for $\rho$. This concludes the proof.
 \end{proof}
We get the following as a corollary:
\begin{cor}\label{cor:so*}
  Let $\rho:\frsu(1,2)\oplus\frsu(1,1)\to \frso^*(4p+2)$ be a tight homomorphism. Then either $\rho$ factors through a homomorphism of $\frsu(1,1)$ or it splits as the direct sum of the tight homomorphisms $\rho_1\oplus\rho_2:\frsu(1,2)\oplus\frsu(1,1)\to\frso^*(6)\oplus\frso^*(4(p-1))$. In each case it is holomorphic in the first factor.
\end{cor}
\begin{proof}
 If we consider the decomposition of $\rho$ in irreducible representations given by Proposition \ref{prop:2.2} we get that the image of $\rho$ is contained in a holomorphically and isometrically embedded subalgebra $\frh$ of the form $\frso^*(2p_1)\oplus\ldots\oplus\frso^*(2p_k)$. If the homomorphism $\rho$ is tight, in particular the inclusion of $\frh$ in $\frso^*(4p+2)$ needs to be tight, and hence all $p_i$ but at most one are even, for rank reasons. Since we know from Proposition \ref{prop:splitting3} that each irreducible homomorphism factors either through $\frsu(1,1)$ or $\frsu(1,2)$, the latter can happen only if one of the $p_i$ is odd. Since in that case the homomorphism $\rho:\frsu(1,2)\to\frso^*(4k+2)$ is holomorphic, we get from \cite{Ham1} that $k=1$. This concludes the proof.
\end{proof}
We finish the section with a lemma that will be useful in the next section.
\begin{lem}\label{lem:e}
 Let $\rho:\frsu(1,1)^2\to\fre_{6(-14)}$ be an injective tight positive homomorphism. Then $\rho$ is the inclusion of a polydisc, in particular it is holomorphic.
\end{lem}
\begin{proof}
 Since $\rho$ is tight, its image is contained in a maximal tube type subdomain of $\fre_{6(-14)}$ that is a Hermitian Lie algebra isomorphic to $\frso(2,8)$, i.e. $\rho$ factors as $\rho=\iota\circ\rho':\frsu(1,1)^2\to\frso(2,8)\to\fre_{6(-14)}$. By Lemma \ref{lem:hulhol} we have that $\Hh(\rho)=\Hh(\iota\circ\rho')=\Hh(\rho')$.
In order to conclude, it is enough to show that the Hermitian hull $\Hh(\rho')$ is already the algebra $\frsu(1,1)^2$. The spin representation $\tau:\frso(2,8)\to \frsu(8,8)$ is holomorphic and tight, hence it follows from Lemma \ref{lem:hulhol} that $\Hh(\rho')=\Hh(\tau\circ\rho')$. It follows from Proposition \ref{prop:2.2} that the homomorphism $\tau\circ \rho'$ is a direct sum of tight homomorphisms factoring through the two factors, and  we get from \cite[Theorem 10]{tight} that $\Hh(\rho')=\Hh(\tau\circ \rho')=\oplus \frsp(2r_i,\R)$. Since the real rank of $\frso(8,2)$ is 2 and the homomorphism $\rho'$ is injective, we get, for rank reasons, that $\Hh(\rho')=\Hh(\rho)$ is isomorphic to $\frsu(1,1)^2$, and this concludes the proof.
\end{proof}

\section{Proofs of the main theorems}\label{sec:proofs}
\begin{proof}[Proof of Theorem \ref{thm:1}]
Using Proposition \ref{prop:tighteasy} we can reduce the problem to the case when $\frh$ simple and $\rho$ is injective.
Our first observation is that, if $\rho:\frg \to \frh$ is an injective homomorphism between Hermitian Lie algebras, then the real rank of $\frh$ is at least as big as the real rank of $\frg$. This can easily be checked by looking at the associated map between the symmetric spaces. In particular if $\frh$ is $\fre_{7(-25)}$, and $\rho$ is tight, then $\frg$ must be of tube type. Since by assumption $\frg$ has no $\frsu(1,1)$ factor, $\frg$ needs to be simple. This implies that $\rho$ is holomorphic by \cite{Ham3}. Similarly, if $\frh=\fre_{6(-14)}$, then either $\frg$ is simple and $\rho$ is holomorphic by \cite{Ham2,Ham3}, or $\frg$ is a product of two rank one algebras. In this case we can consider a tight embedding $d:\frsu(1,1)^2\to\frg$ and deduce from Lemma \ref{lem:e} that $\rho$ is simple.

Next we assume that $\frh$ is classical.
 Let $\frak g=\frak g_1\oplus\ldots\oplus \frak g_k$ be the decomposition of $\frak g$ into simple factors. In order to show that $\rho$ is holomorphic, it is enough to find, for every simple factor $\frg_s$, a holomorphically embedded subalgebra $j_s:\frak l_s\to \frak g_s$ such that the composition $\rho\circ j_s:\frak l_s\to \frh $ is holomorphic. Once that is done, it is easy to conclude that the restriction of $\rho$ to $\frak g_s$ is holomorphic, using the irreducibility of $\frg_s$ under the adjoint action of $\frg_s$ on itself  (see \cite[Lemma 5.4]{Ham2}). Clearly if the restriction of $\rho$ to each simple factor is holomorphic, the same is true for the homomorphism $\rho$.
 
Let us now fix a factor $\frg_s$. It is easy to check with a case by case argument that there exists a tight injective and holomorphic embedding $j_s:\frl_s\to \frg_s$ where $\frl_s$ is either $\frsu(1,2)$ in case $\frg_s=\frsu(1,p)$, $\frsp(4,\R)$ in case $\rk (\frg_s)$ is even or $\frsp(4,\R)\oplus \frsu(1,1)$ in case $\rk (\frg_s)$ is odd and greater than one (see \cite[Theorem 7.1]{Ham2}). 

Let us first deal with the factors $\frg_s$ of higher rank. We consider the tight holomorphic embedding $\phi_s:\frsp(4,\R)\oplus \frsu(1,1)\to \frg$ that splits as a sum of a diagonal disc $d_t:\frsu(1,1)\to \frg_t$ if $t$ is different from $s$ and the holomorphic tight embedding $j_s$ in the factor $\frg_s$. An easy computation confirms that $\phi_s$ is tight and injective, hence the composition $\rho\circ \phi_s$ is tight. Since $\frsp(4,\R)\oplus\frsu(1,1)$ is of tube type, we get, using Proposition \ref{prop:tt} that the image of $\rho\circ\phi_s$ is contained in a tube type subdomain, that hence admits a tight holomorphic embedding in $\frsu(m,m)$ for some $m$. Applying Proposition \ref{prop:hamlowrk} we get that $\rho\circ\phi_s$ is holomorphic in the $\frsp(4,\R)$ factor, hence $\rho$ is holomorphic in the $\frg_s$ factor.

The case in which $\frg_s$ is isomorphic to $\frsu(1,p)$ splits in two cases: if $\frh$ is $\frsu(m,n)$ we apply the same argument combined with Corollary \ref{cor:ham} instead of Proposition \ref{prop:hamlowrk}, if, instead, $\frh$ is $\frso^*(2p)$ we apply Corollary \ref{cor:so*}.
\end{proof}
\begin{proof}[Proof of Theorem \ref{thm:2}]
We first consider the case $\frh=\frsu(p,q)$, or $\frh=\frso^*(2p)$. 
Applying inductively Proposition \ref{prop:splitting2} we can show that the homomorphism $\rho:\frsu(1,1)^l\bigoplus_{i=1}^n\frg_i\rightarrow \frh$ can be decomposed as a direct sum   of homomorphisms $\rho_j$ factoring through the factors: 
$$\bigoplus_{j=1}^{n+l}\rho_j:\frsu(1,1)^l\bigoplus_{i=1}^n\frg_i\rightarrow\bigoplus_{j=1}^{n+l} \frh_j.$$

Denote by $\iota:\oplus_j \frh_j\to \frh$ the tight holomorphic  inclusion. Since $\iota$  is holomorphic, it follows from Lemmas \ref{lem:hulhol} and  \ref{lem:hul} that $\mathcal{H}(\rho)=\mathcal{H}(\iota\circ\oplus\rho_j)=\mathcal{H}(\oplus\rho_j)=\bigoplus \mathcal{H}(\rho_j)$. 
For $j\leq l$ we have that $\rho_j$ is a map $\rho_j:\frsu(1,1)\rightarrow\frh_j$, thus $\mathcal{H}(\rho_j)=\bigoplus_k\frsp(2m_{k,j},\mathbb{R})$ for these $j$ by \cite[Corollary 9.6]{tight}.
For the remaining $j$:s we have maps $\rho_j:\frg_i\rightarrow \frh_j$ where $j=l+i$ and $\frg_i\neq\frsu(1,1)$. These maps must be holomorphic by the results in \cite{Ham2} and \cite{Ham3}. The statement of the theorem thus follows for these cases.

Next we consider the cases $\frh=\frsp(2n,\R),\frso(2,n)$.
All these $\frh$ can be tightly and holomorphically embedded into some $\frsu(n,n)$ by the results in \cite{Ham1}, denote this embedding $\iota$.
By Lemma \ref{lem:hulhol} we have $\mathcal{H}(\rho)\simeq\mathcal{H}(\iota\circ\rho)$. The Hermitian hull must thus be of the prescribed form by the results for the first case.   

We finish the proof with the case $\fre_{6(-14)}$: if $\frg$ is simple, the result follows from \cite{Ham3}, otherwise fix a splitting $\frg=\frg_1\oplus\frg_2$ and consider the composition of $\rho$ with a sum of (anti)-diagonal discs $d:\frsu(1,1)\oplus\frsu(1,1)\to \frg_1\oplus\frg_2$ chosen in such a way that the composition $\rho\circ d$ is tight.
It follows from Lemma \ref{lem:e} that the composition $\rho\circ d$ is holomorphic. This implies that $\rho$ is holomorphic as well: since the rank of $\fre_{6(-14)}$ is two, the algebra $\frg$ has precisely two rank one factors. This implies that $\calH(\rho)\cong \frg$. This concludes the proof.
\end{proof}

\section{Lists}\label{sec:lists}
We start recalling, for the reader's convenience the isomorphisms between Lie algebras of Hermitian type, see \cite[Page 519]{Helgason} for more details:
\begin{center}
\begin{table}[h!]
 \begin{tabular}{|c|c|c|c|c|} 
 \hline $I_{mn}$& $II_{m}$&$III_{m}$&$IV_{m}$&\\\hline
  \hline $\frsu(1,1)$&&$\frsp(2,\R)$&$\frso(2,1)$&\\\hline
  &&&$\frso(2,2)$&$\frsu(1,1)\oplus\frsu(1,1)$\\\hline
  &$\frso^*(4)$&&&$\frsu(1,1)\oplus \frsu(2)$\\\hline
  $\frsu(1,3)$&$\frso^*(6)$&&&\\\hline
  &&$\frsp(4,\R)$&$\frso(2,3)$&\\\hline
  $\frsu(2,2)$&&&$\frso(2,4)$&\\\hline
  &$\frso^*(8)$&&$\frso(2,6)$&\\\hline
  
 \end{tabular}\\[5pt]\caption{}
\end{table}
\end{center}
Due to the isomorphism $\frsu(m,n)\cong \frsu(n,m)$ we will only consider the algebras $\frsu(m,n)$ with $m\leq n$.
\subsection{Some building blocks}
\subsubsection{Inclusions of higher rank domains}
We fix the following matrix models:
$$\begin{array}{cl}\frsu(p,q)&=\left\{ \bpm A&Z\phantom{^*}\\Z^{*}&B\phantom{^*}\epm,\, \begin{array}{l}A\in M_p(\C),\; B\in M_q(\C),\; Z\in M_{p,q}(\C),\\
 A^{*}=-A,\;B^{*}=-B,\; \mbox{tr}(A)+\mbox{tr}(B)=0\end{array}\right\},\\\\
\frsp(2n,\R)&=\left\{ \bpm A&Z\\Z^{*}&A^{*}\epm,\, \begin{array}{l}A\in M_n(\C),\; Z\in M_{n}(\C),\\ A^{*}=-A,\;Z^T=Z\end{array}\right\},\\\\
\frso^{*}(2n)&=\left\{ \bpm A&Z\\Z^{*}&A^{*}\epm,\,\begin{array}{l} A\in M_n(\C),\; Z\in M_{n}(\C),\\ A^{*}=-A,\;Z^T=-Z\end{array}\right\}.\end{array}$$

We also fix the classical realizations for the associated symmetric spaces that can be found, for example in  \cite[Pages 91, 115 and 123]{PS}. These are classically referred to as the Harish-Chandra realization of the symmetric spaces as a bounded domain in a complex Euclidean space.
 $$\begin{array}{cl}
\Xx_{\frsu(p,q)}&=\{Z\in M_{p,q}(\C)|\; \Id_p-Z^*Z<0\};\\
\Xx_{\frsp(2n,\R)}&=\{Z\in M_n(\C)|\; \Id_n-Z^*Z<0,\; Z^T=Z\};\\
\Xx_{\frso^*(2n)}&=\{Z\in M_n(\C)|\; \Id_n-Z^*Z<0,\;Z^T=-Z\}.\\    
   \end{array}
$$ 
 
Using the explicit realizations it is easy to describe some holomorphic maps between these spaces and the corresponding Lie algebra homomorphisms. We have the obvious inclusion homomorphisms which corresponds to the natural inclusion maps for the realizations we have chosen:
$$
\begin{array}{rclccrcl}\frsp(2n,\R)&\rightarrow&\frsu(n,n),&&& \Xx_{\frsp(2n,\R)}&\rightarrow &\Xx_{\frsu(n,n)}\\
\frso^{*}(2n)&\rightarrow&\frsu(n,n),&&& \Xx_{\frso^*(2n)}&\rightarrow& \Xx_{\frsu(n,n)}
\end{array}
$$
The homomorphisms in the first class are tight for all $n$, the homomorphisms in the second class for even $n$: this is because, with the given normalization of the metric, the inclusion of $\Xx_{\frsp(2n,\R)}$ in $\Xx_{\frsu(n,n)}$ is isometric and the two Lie algebras have the same rank, and the inclusion  $\tau:\Xx_{\frso^*(2n,\R)}\to\Xx_{\frsu(n,n)}$ satisfies $\tau^*(\o_{\frsu(n,n)})=2\o_{\frso^*(2n)}$, but the rank of $\Xx_{\frso^*(2n)}$ is $\lfloor n/2\rfloor$.

There are also holomorphic homomorphisms from $\frsu(m,n)$ to the algebras $\frsp(2(m+n),\R)$ and $\frso^*(2(m+n))$, we describe them,  here below:

$$\begin{array}{cccccc}
\frsu(m,n)&\rightarrow&\frsp(2(m+n),\R)& \frsu(m,n)&\rightarrow&\frso^*(2(m+n)),\\
\bpm A&Z\\Z^{*}&B\epm&\mapsto &\bpm A&0&0&Z\\0&\bar B&Z^T&0\\ 0&\bar{Z}&\bar A&0\\Z^*&0&0&B  \epm\phantom{a}&
\bpm A&Z\\Z^{*}&B\epm&\mapsto&\bpm A&0&0&Z\\0&\bar B&-Z^T&0\\ 0&-\bar{Z}&\bar A&0\\Z^*&0&0&B  \epm.
\end{array}$$
They correspond to the following holomorphic and isometric maps between the associated symmetric spaces:
$$
\begin{array}{ccccccc}
\Xx_{\frsu(m,n)}&\rightarrow &\Xx_{\frsp(2(m+n),\R)} && \Xx_{\frsu(m,n)}&\rightarrow &\Xx_{\frso^*(2(m+n))},\\
 \phantom{\Bigg(}Z&\mapsto &\bpm0&Z\\Z^T&0\epm&& Z&\mapsto& \bpm0&Z\\-Z^T&0\epm.\end{array}
$$
Comparing the ranks one easily gets that the homomorphisms in the first class are tight precisely when $m=n$, and the homomorphisms in the second when $n=m$ or $m+1$ (see \cite{Ham1} for more details). 

It is easy to check that the compositions of the tight embeddings described above, namely
\begin{center}
\begin{tabular}{l} 
$\frsp(2m,\R)\to\frsu(m,m)\to\frsp(4m,\R)$\\ $\frsu(m,m)\to\frsp(4m,\R)\to\frsu(2m,2m)$\\
$\frso^*(2m)\to\frsu(m,m)\to\frso^*(4m)$\\ $\frsu(m,m)\to\frso^*(4m)\to\frsu(2m,2m)$ 
\end{tabular}
\end{center}
 are conjugate to diagonal inclusions.

\subsubsection{The spin representations}
Let $V$ be a real vector space endowed with an orthogonal form $S$ of signature $(2,p)$. The Clifford algebra $\calC$ over the complexification $V_\C$ is the quotient of the tensor algebra $\bigoplus_r V_\C^{\otimes r}$ by the ideal generated by elements of the form $x\otimes x-S(x,x)$. The Clifford algebra is a complex vector space of dimension $2^{p+2}$ and it is shown in \cite[Page 453]{Satake-class} that its half dimensional ideal $\calC^+$ consisting of even degree elements can be realized, depending on the parity of $p$, as 
$$\begin{array}{ll}
   M_{2^{k+1}}(\C)& p=2k+1\\
   M_{2^{k}}(\C)\times M_{2^{k}}(\C)&  p=2k. 
  \end{array}
$$
It is well known that the twofold covering group ${\rm Spin}(2,p)$  of $\SO(s,m)$ has a natural realization as a subgroup of $\calC^+$. Moreover Satake, in the paper cited above, showed that there exists a natural Hermitian form on $\C^{2^{k+1}}$ (resp. $\C^{2^{k}}$) of signature $(2^{k},2^{k})$ (resp. $(2^{k-1},2^{k-1})$)  and the group ${\rm Spin}(2,p)$ is contained in $\SU(2^{k},2^{k})$ (resp. $(\SU(2^{k-1},2^{k-1}))^2$). This gives a homomorphism $\rho_{2k+1}$ of $\frso(2,2k+1)$ into $\frsu(2^{k},2^{k})$ (resp. two inequivalent homomorphisms $\rho_{2k}$ and $\rho'_{2k}$ of $\frso(2,2k)$ into $\frsu(2^{k-1},2^{k-1})$). The image of these homomorphisms are contained in the subalgebra $\frso^*(4l)$ if  $p=1,2,3$ modulo 8 and in $\frsp(2l,\R)$ if $p=5,7,8$ modulo 8, where $l$ depends on $p$ as in the table below. We summarize what we just recalled about the spin representations in this table:

\begin{center}
\begin{table}[h!]
\begin{tabular}{|c|c|crcl|} \hline $p\;(8)$&$l_p$&& Spin&&\\
\hline
    $1,3$ &$2^{(p-3)/_2}$& $\rho_{2k+1}:$&$\frso(2,2k+1)$&$\to $&$\frso^*(2^{k+1})\subset \frsu(2^{k},2^{k})$  \\
    $2$&$2^{p/_2-2} $ & $\rho_{2k},\,\rho'_{2k} :$&$\frso(2,2k)$&$\to $&$\frso^*(2^{k})\subset \frsu(2^{k-1},2^{k-1})$  \\
    $4,6$ &$2^{p/_2-1}$& $\rho_{2k},\,\rho'_{2k} :$&$\frso(2,2k)$&$\to$&$ \frsu(2^{k-1},2^{k-1})$  \\
    $5,7$ &$2^{(p-1)/_2}$& $\rho_{2k+1}:$&$\frso(2,2k+1)$&$\to $&$\frsp(2^{k+1},\R)\subset \frsu(2^{k},2^{k})$  \\
    $8$ &$2^{p/_2-1}$& $\rho_{2k},\,\rho'_{2k} :$&$\frso(2,2k)$&$\to $&$\frsp(2^{k},\R)\subset \frsu(2^{k-1},2^{k-1})$  \\\hline
  
\end{tabular}\\[5pt]\caption{}\end{table}   \end{center}
It was checked in \cite{Ham1} that every spin representation is tight and holomorphic.

%
\subsubsection{The non-holomorphic tight representations}
The irreducible representations of odd highest weight $\rho_n^\C:\frsl(2,\C)\rightarrow \frgl(V)=\frgl(2n,\C)$ restricts to tight homomorphisms $\rho_n:\frsu(1,1)\rightarrow\frsp(2n,\R)$ that are non-holomorphic when $n>1$. These representations are a building block of any non-holomorphic representation. For an explicit construction of these representations we refer to \cite[Example 8.8]{BIW}.

\subsection{Lists of tight homomorphisms}

\begin{center}
\begin{tabular}{|c|}  \hline
     $\rho:\frg\to \frsu(m,m)$  
     \\\hline
\end{tabular}   \end{center}
Every tight homomorphism $\rho:\frg\to \frsu(m,m)$ must fulfill the equation at the bottom row of the table. The homomorphism is described below in terms of the homomorphisms defined earlier in the section.
\vspace{.3cm}

\hspace{-1.1cm}\vspace{.5cm}
\begin{table}[h!]
\begin{tabular}{|c|c|}
 \hline
 $\frg$& \parbox{200pt}{$$\frsu(1,1)^A\bigoplus_{i=1}^{ I}\frsu(p_i,p_i)\bigoplus_{l=1}^{ L}\frsp(2n_l,\R)\bigoplus_{j=1}^{ J}\frso^*(4m_j)\bigoplus_{s=1}^{ S}
 \frso(2,r_s)$$}\\\hline
 
 $\calH(\rho)$& \parbox{200pt}{$$\bigoplus_{a=1}^A\bigoplus_{b=1}^{ B_a}\frsp(2f_{ab},\R)\bigoplus_{i=1}^{ I}\frsu(p_i,p_i)\bigoplus_{l=1}^{ L}\frsp(2n_l,\R)\bigoplus_{j=1}^{ J}\frso^*(4m_j)\bigoplus_{s=1}^{ S}
 \frso(2,r_s)$$}
 \\\hline
 
 $\frg_{\rm reg}$&\parbox{200pt}{$$\bigoplus_{a=1}^A\bigoplus_{b=1}^{ B_a}\frsu(f_{ab},f_{ab})^{g_{1ab}}\bigoplus_{i=1}^{ I}\frsu(p_i,p_i)^{g_{2i}}\bigoplus_{l=1}^{ L}\frsu(n_l,n_l)^{g_{3l}} $$}\\   &\parbox{300pt}{$$   \bigoplus_{j=1}^{ J}\frsu(2m_j,2m_j)^{g_{4j}}\bigoplus_{s=1}^{ S}
 \frsu(
 {l_{r_s}},{l_{r_s}}
 )^{g_{5s}}$$}

 \\ \hline
    &\parbox{300pt}{$$\sum_{a,b}g_{1ab}f_{ab}+\sum_i g_{2i}p_i+ \sum_l g_{3l}n_l+\sum_j g_{4j}m_j+\sum_s g_{5s} l_{r_s}=m$$}\\\hline
\end{tabular}\\[5pt]
{\caption{}}\end{table}

Using the matrix models described above the image under the homomorphism $\rho$ of the element   $(X_1,..,X_A,Y_1,...,Y_{I},U_1,...,U_{L},V_1,...,V_{J},W_1,...,W_{S})=(X,Y,U,V,W)$ of $\frg$ is the matrix
\begin{center}$
\bpm A_1 &0 &0 &Z_1 &0 &0\\
0&\ddots&0&0&\ddots&0\\
0&0&A_5&0&0&Z_5\\
Z_1^* &0 &0 &B_1 &0 &0\\
0&\ddots&0&0&\ddots&0\\
0&0&Z_5^*&0&0&B_5
\epm$  where \end{center}

\hspace{-2.3cm}
 \makebox{$\bpm A_1 &Z_1\\ Z_1^*&B_1\epm
=\bpm A_{11} &0 &0 &Z_{11} &0 &0\\
0&\ddots&0&0&\ddots&0\\
0&0&A_{1A}&0&0&Z_{1A}\\
Z_{11}^* &0 &0 &B_{11} &0 &0\\
0&\ddots&0&0&\ddots&0\\
0&0&Z_{1A}^*&0&0&B_{1A}
\epm$,
 $\bpm A_{1a} &Z_{1a}\\ Z_{1a}^*&B_{1a}\epm
=\bpm A_{1a1} &0 &0 &Z_{1a1} &0 &0\\
0&\ddots&0&0&\ddots&0\\
0&0&A_{1aB_a}&0&0&Z_{1aB_a}\\
Z_{1a1}^* &0 &0 &B_{1a1} &0 &0\\
0&\ddots&0&0&\ddots&0\\
0&0&Z_{1aB_a}^*&0&0&B_{1aB_a}
\epm$,}

\vspace{.3cm}
\hspace{-1.5cm}
\makebox{$\bpm A_{1ab} &Z_{1ab}\\ Z_{1ab}^*&B_{1ab}\epm=
\bpm A_{1ab1} &0 &0 &Z_{1ab1} &0 &0\\
0&\ddots&0&0&\ddots&0\\
0&0&A_{1abg_{1ab}}&0&0&Z_{1abg_{1ab}}\\
Z_{1ab1}^* &0 &0 &B_{1ab1} &0 &0\\
0&\ddots&0&0&\ddots&0\\
0&0&Z_{1abg_{1ab}}^*&0&0&B_{1abg_{1ab}}
\epm$,\phantom{ciao}
$\bpm A_{1abc} &Z_{1abc}\\ Z_{1abc}^*&B_{1abc}\epm
=\rho_{f_{ab}}(X_a)$;}

\vspace{.2cm}
\hspace{-2.3cm}
\makebox{
$\bpm A_2 &Z_2\\ Z_2^*&B_2\epm
=\bpm A_{21} &0 &0 &Z_{21} &0 &0\\
0&\ddots&0&0&\ddots&0\\
0&0&A_{2I}&0&0&Z_{2I}\\
Z_{21}^* &0 &0 &B_{21} &0 &0\\
0&\ddots&0&0&\ddots&0\\
0&0&Z_{2I}^*&0&0&B_{2I}
\epm$, 
$\bpm A_{2i} &Z_{2i}\\ Z_{2i}^*&B_{2i}\epm
=\bpm A_{2i1} &0 &0 &Z_{2i1} &0 &0\\
0&\ddots&0&0&\ddots&0\\
0&0&A_{2ig_{2i}}&0&0&Z_{2ig_{2i}}\\
Z_{2i1}^* &0 &0 &B_{2i1} &0 &0\\
0&\ddots&0&0&\ddots&0\\
0&0&Z_{2ig_{2i}}^*&0&0&B_{2ig_{2i}}
\epm$,}\\[5pt]
\begin{center}
$\bpm A_{2i1} &Z_{2i1}\\ Z_{2i1}^*&B_{2i1}\epm= ...= \bpm A_{2ig_{2i}} &Z_{2ig_{2i}}\\ Z_{2ig_{2i}}^*&B_{2ig_{2i}}\epm
=Y_i$;
\end{center}

\vspace{10pt}
\hspace{-2.3cm}
\makebox{
$\bpm A_3 &Z_3\\ Z_3^*&B_3\epm
=\bpm A_{31} &0 &0 &Z_{31} &0 &0\\
0&\ddots&0&0&\ddots&0\\
0&0&A_{3L}&0&0&Z_{3L}\\
Z_{31}^* &0 &0 &B_{31} &0 &0\\
0&\ddots&0&0&\ddots&0\\
0&0&Z_{3L}^*&0&0&B_{3L}
\epm$,  
$\bpm A_{3l} &Z_{3l}\\ Z_{3l}^*&B_{3l}\epm
=\bpm A_{3l1} &0 &0 &Z_{3l1} &0 &0\\
0&\ddots&0&0&\ddots&0\\
0&0&A_{3lg_{3l}}&0&0&Z_{3lg_{3l}}\\
Z_{3l1}^* &0 &0 &B_{3l1} &0 &0\\
0&\ddots&0&0&\ddots&0\\
0&0&Z_{3lg_{3l}}^*&0&0&B_{3lg_{3l}}
\epm$,}\\[5pt]
\begin{center}
$\bpm A_{3l1} &Z_{3l1}\\ Z_{3l1}^*&B_{3l1}\epm=...=\bpm A_{3lg_{3l}} &Z_{3lg_{3l}}\\ Z_{3lg_{3l}}^*&B_{3lg_{3l}}\epm
=U_l$;
\end{center}

\vspace{10pt}
\hspace{-2.3cm}
\makebox{
$\bpm A_4 &Z_4\\ Z_4^*&B_4\epm
=\bpm A_{41} &0 &0 &Z_{41} &0 &0\\
0&\ddots&0&0&\ddots&0\\
0&0&A_{4J}&0&0&Z_{4J}\\
Z_{41}^* &0 &0 &B_{41} &0 &0\\
0&\ddots&0&0&\ddots&0\\
0&0&Z_{4J}^*&0&0&B_{4J}
\epm$, 
$\bpm A_{4j} &Z_{4j}\\ Z_{4j}^*&B_{4j}\epm
=\bpm A_{4j1} &0 &0 &Z_{4j1} &0 &0\\
0&\ddots&0&0&\ddots&0\\
0&0&A_{4jg_{4j}}&0&0&Z_{4jg_{4j}}\\
Z_{4j1}^* &0 &0 &B_{4j1} &0 &0\\
0&\ddots&0&0&\ddots&0\\
0&0&Z_{4jg_{4j}}^*&0&0&B_{4jg_{4j}}
\epm$,}\\[5pt]
\begin{center}
$\bpm A_{4j1} &Z_{4j1}\\ Z_{4j1}^*&B_{4j1}\epm=...=\bpm A_{4jg_{4j}} &Z_{4jg_{4j}}\\ Z_{4jg_{4j}}^*&B_{4jg_{4j}}\epm
=V_j$;
\end{center}

\vspace{10pt}
\hspace{-2.3cm}
\makebox{
$\bpm A_5 &Z_5\\ Z_5^*&B_5\epm
=\bpm A_{51} &0 &0 &Z_{51} &0 &0\\
0&\ddots&0&0&\ddots&0\\
0&0&A_{5S}&0&0&Z_{5S}\\
Z_{51}^* &0 &0 &B_{51} &0 &0\\
0&\ddots&0&0&\ddots&0\\
0&0&Z_{5|S|}^*&0&0&B_{5|S|}
\epm$, 
$\bpm A_{5s} &Z_{5s}\\ Z_{5s}^*&B_{5s}\epm
=\bpm A_{5s1} &0 &0 &Z_{5s1} &0 &0\\
0&\ddots&0&0&\ddots&0\\
0&0&A_{5sg_{5s}}&0&0&Z_{5sg_{5s}}\\
Z_{5s1}^* &0 &0 &B_{5s1} &0 &0\\
0&\ddots&0&0&\ddots&0\\
0&0&Z_{5sg_{5s}}^*&0&0&B_{5sg_{5s}}
\epm$, }\\[5pt]
\begin{center}
$\bpm A_{5st} &Z_{5st}\\ Z_{5st}^*&B_{5st}\epm=\rho_{r_s}(W_j)$ (or $\rho'_{r_s}(W_j)$ if $r_s$ is even) for $t=1,...,g_{5s}$.
\end{center}

\vspace{20pt}
\begin{center}
\begin{tabular}{|c|}  \hline
     $\rho:\frg\to \frsu(m,n)$ \\ 
             \hline
\end{tabular}   \end{center}
If $\rho:\frg\to\frsu(m,n)$ is a tight homomorphism and we denote by $\frg^T$ and $\frg^{NT}$ the subalgebras of $\frg$ consisting of the product of all the factors of tube type (resp. not of tube type), it follows from Proposition \ref{prop:splitting2} that there exists $k$ and a subalgebra of $\frsu(m,n)$ of the form $\frsu(k,k)\oplus\frsu(m-k,n-k)$ such that the homomorphism $\rho$ splits as direct sum of $\rho^T:\frg^T\to \frsu(k,k)$ and $\rho^{NT}:\frg^{NT}\to \frsu(m-k,n-k)$. In particular $\rho^T$ is one of the homomorphisms described in Table 3, and we have:

\begin{table}[h!]
\begin{tabular}{|c|c|}
 \hline
 $\frg$& \parbox{200pt}{$$\frg^T\bigoplus_{i=1}^{ I}\frsu(p_i,q_i)\phantom{more space more space}$$}\\\hline
 $\calH(\rho)$& \parbox{200pt}{$$\calH(\rho^T)\bigoplus_{i=1}^{ I}\frsu(p_i,q_i)\phantom{more space more space}$$}
 \\\hline
 $\frg_{\rm reg}$&\parbox{200pt}{$$\frg_{\rm reg}^T\bigoplus_{i=1}^{ I}\frsu(p_i,q_i)^{g_i}\phantom{more space more space}$$}
 \\ \hline
    &\parbox{200pt}{$$k+\sum_ig_ip_i=m\phantom{more space more space}$$}\\\hline
\end{tabular}\\[5pt]
\caption{}
\end{table}

An explicit description of the homomorphism is given as follows. An element $(X,X_1,...,X_{I})\in \frg^T\bigoplus_{i\in I} \frsu(p_i,q_i)$ is mapped as:\begin{center}
$(X,X_1,...,X_{I})\mapsto \bpm A_1 &0  &Z_1 &0 \\
0&A_2&0&Z_2\\
Z_1^* &0  &B_1 &0\\
0&Z_2^*&0&B_2
\epm$, where 
$\bpm A_1 &Z_1\\ Z_1^*&B_1\epm
=\rho^T(X)$, and\\
\end{center}
\hspace{-2.3cm}
\makebox{
$\bpm A_2 &Z_2\\ Z_2^*&B_2\epm
=\bpm A_{21} &0 &0 &Z_{21} &0 &0\\
0&\ddots&0&0&\ddots&0\\
0&0&A_{2I}&0&0&Z_{2I}\\
Z_{21}^* &0 &0 &B_{21} &0 &0\\
0&\ddots&0&0&\ddots&0\\
0&0&Z_{2I}^*&0&0&B_{2I}
\epm$, 
$\bpm A_{2i} &Z_{2i}\\ Z_{2i}^*&B_{2i}\epm
=\bpm A_{2i1} &0 &0 &Z_{2i1} &0 &0\\
0&\ddots&0&0&\ddots&0\\
0&0&A_{2ig_i}&0&0&Z_{2ig_i}\\
Z_{2i1}^* &0 &0 &B_{2i1} &0 &0\\
0&\ddots&0&0&\ddots&0\\
0&0&Z_{2ig_i}^*&0&0&B_{2ig_i}
\epm$, }
\begin{center}
$\bpm A_{2i1} &Z_{2i1}\\ Z_{2i1}^*&B_{2i1}\epm= ...= \bpm A_{2ig_i} &Z_{2ig_i}\\ Z_{2ig_i}^*&B_{2ig_i}\epm
=X_i$.

\end{center}

\begin{center}
\begin{tabular}{|c|}  \hline
     $\rho:\frg\to \frso^*(4p)$ \\ 
             \hline
\end{tabular}   \end{center}

If $\rho:\frg\to\frso^*(4p)$ is a tight homomorphism and we denote by $\frg_1$ (resp. $\frg_2$) the subalgebras of $\frg$ consisting of the product of all the factors non isomorphic to $\frso^*(2k)$ nor to $\frso(2,s)$  for any $k$ (resp. isomorphic to), there exists $k$ and a subalgebra of $\frso^*(4p)$ of the form $\frso^*(4k)\oplus\frso^*(4p-4k)$ such that the homomorphism $\rho$ splits as direct sum of $\rho_1:\frg_1\to \frsu(k,k)\to\frso^*(4k)$ and $\rho_2:\frg_2\to \frso^*(4p-4k)$. In particular $\rho_1$ is composition of one of the homomorphisms described in Table 3 and the standard embedding of $\frsu(k,k)$ in $\frso^*(4k)$ and we have:\vspace{.3cm}

\hspace{-1cm}
\begin{table}[h!]
\begin{tabular}{|c|c|}
 \hline
 $\frg$& \parbox{350pt}{$$\frg_1\bigoplus_{i=1}^{ I}\frso^*(4m_i)\bigoplus_{j=1}^{ J}
 \frso(2,r_j)$$}\\\hline
 $\calH(\rho)$& \parbox{350pt}{$$\calH(\rho_1)\bigoplus_{i=1}^{ I}\frso^*(4m_i)\bigoplus_{j=1}^{ J}
 \frso(2,r_j)$$}
 \\\hline
 $\frg_{\rm reg}$&\parbox{350pt}{$$\frg_{\rm reg}^1\bigoplus_{i=1}^{ I}\frso^*(4m_i)^{g_{2i}}\bigoplus_{j=1}^{ J}\frso^*
 (4l_{r_j})^{g_{3j}}$$}
 \\ \hline
    &\parbox{350pt}{$$k+\sum_ig_{2i}m_i+\sum_jg_{3j} l_{r_j}
    =p,\; r_j\equiv 5,6,7 \,(8)$$}\\\hline
\end{tabular}\\[5pt]\caption{}\end{table}

\vspace{.3cm}

An explicit description of the homomorphism is given as follows. The image of the  element $(X,X_1,...,X_I,Y_1,...,Y_J)$ of  $\frg_1 \bigoplus_{i=1}^{ I} \frso^*(4m_i)\bigoplus_{j=1}^{ J} \frso(2,r_j)$ is the matrix:
\begin{center}
$
\bpm A_1 &0 &0 &Z_1 &0 &0\\
0&A_2&0&0&Z_2&0\\
0&0&A_3&0&0&Z_3\\
Z_1^* &0  &0&B_1 &0&0\\
0&Z_2^*&0&0&B_2&0\\
0&0&Z_3^*&0&0&B_3
\epm$, where 
$\bpm A_1 &Z_1\\ Z_1^*&B_1\epm
=\rho_1(X)$,\\
\end{center}
\hspace{-2.3cm}
\makebox{
$\bpm A_2 &Z_2\\ Z_2^*&B_2\epm
=\bpm A_{21} &0 &0 &Z_{21} &0 &0\\
0&\ddots&0&0&\ddots&0\\
0&0&A_{2I}&0&0&Z_{2I}\\
Z_{21}^* &0 &0 &B_{21} &0 &0\\
0&\ddots&0&0&\ddots&0\\
0&0&Z_{2I}^*&0&0&B_{2I}
\epm$, 
$\bpm A_{2i} &Z_{2i}\\ Z_{2i}^*&B_{2i}\epm
=\bpm A_{2i1} &0 &0 &Z_{2i1} &0 &0\\
0&\ddots&0&0&\ddots&0\\
0&0&A_{2ig_{2i}}&0&0&Z_{2ig_{2i}}\\
Z_{2i1}^* &0 &0 &B_{2i1} &0 &0\\
0&\ddots&0&0&\ddots&0\\
0&0&Z_{2ig_{2i}}^*&0&0&B_{2ig_{2i}}
\epm$, }
\begin{center}
$\bpm A_{2i1} &Z_{2i1}\\ Z_{2i1}^*&B_{2i1}\epm= ...= \bpm A_{2ig_{2i}} &Z_{2ig_{2i}}\\ Z_{2ig_{2i}}^*&B_{2ig_{2i}}\epm
=X_i$,
\end{center}
\hspace{-2.3cm}
\makebox{
$\bpm A_3 &Z_3\\ Z_3^*&B_3\epm
=\bpm A_{31} &0 &0 &Z_{31} &0 &0\\
0&\ddots&0&0&\ddots&0\\
0&0&A_{3J}&0&0&Z_{3J}\\
Z_{31}^* &0 &0 &B_{31} &0 &0\\
0&\ddots&0&0&\ddots&0\\
0&0&Z_{3J}^*&0&0&B_{3J}
\epm$, 
$\bpm A_{3j} &Z_{3j}\\ Z_{3j}^*&B_{3j}\epm
=\bpm A_{3j1} &0 &0 &Z_{3j1} &0 &0\\
0&\ddots&0&0&\ddots&0\\
0&0&A_{3jg_{3j}}&0&0&Z_{3jg_{3j}}\\
Z_{3j1}^* &0 &0 &B_{3j1} &0 &0\\
0&\ddots&0&0&\ddots&0\\
0&0&Z_{3jg_{3j}}^*&0&0&B_{3jg_{3j}}
\epm$,}
\begin{center}
$\bpm A_{3jt} &Z_{3jt}\\ Z_{3jt}^*&B_{3jt}\epm=\rho_{r_j}(Y_j)$ (or $\rho'_{r_j}(Y_j)$ if $r_j$ is even) for $t=1,...,g_{3j}$.
\end{center}

\begin{center}
\begin{tabular}{|c|}  \hline
     $\rho:\frg\to \frso^*(4p+2)$ \\ 
             \hline
\end{tabular}   \end{center}

In this case either $\frg$ is of tube type, and $\rho(\frg)$ is contained in a subalgebra of the form $\frso^*(4p)$ or $\frg$ has precisely a simple factor that is not of tube type and is isomorphic either to $\frso^*(4k+2)$ or to $\frsu(k,k+1)$. Denoting by $\frg^1$ the product of all the tube type factor of $\frg$ that are of tube type, it follows from Proposition \ref{prop:splitting3} that $\rho$ splits as the direct sum of a tight homomorphism $\rho_1:\frg^1\to\frso^*(4(p-k))$ and either the isomorphism $\frso^*(4k+2)$ with a subalgebra of $\frso^*(4p+2)$ or the inclusion $\frsu(k,k+1)$ into a subalgebra of the form $\frso^*(4k+2)$.

\begin{center}
\begin{tabular}{|c|}  \hline
     $\rho:\frg\to \frsp(2p,\R)$ \\ 
             \hline
\end{tabular}   \end{center}

If $\rho:\frg\to\frsp(2p,\R)$ is a tight homomorphism and we denote by $\frg_1$ (resp. $\frg_2$) the subalgebras of $\frg$ consisting of the product of all the factors non isomorphic to $\frsp(2k,\R)$ nor to $\frso(2,s)$  for any $s$ (resp. isomorphic to), there exists $k$ and a subalgebra of $\frsp(2p,\R)$ of the form $\frsp(2k,\R)\oplus\frsp(2(p-k),\R)$ such that the homomorphism $\rho$ splits as direct sum of $\rho_1:\frg_1\to \frsu(k,k)\to\frsp(2k,\R)$ and $\rho_2:\frg_2\to \frsp(2(p-k),\R)$. In particular $\rho_1$ is composition of one of the homomorphisms described in Table 3 and the standard embedding of $\frsu(k,k)$ in $\frsp(2k,\R)$ and we have:

\hspace{-1cm}
\begin{center}
\begin{table}[h]
\begin{tabular}{|c|c|}
 \hline
 $\frg$&\hspace{-3cm}\parbox{300pt}{$$\frg^1\bigoplus_{i=1}^{ I}\frsp(2m_i,\R)\bigoplus_{j=1}^{ J}
 \frso(2,r_j)$$}\\\hline
 $\calH(\rho)$&\hspace{-3.5cm} \parbox{300pt}{$$\calH(\rho_1)\bigoplus_{i=1}^{ I}\frsp(2m_i,\R)\bigoplus_{j=1}^{ J}
 \frso(2,r_j)$$}
 \\\hline
 $\frg_{\rm reg}$&\hspace{-3cm}\parbox{300pt}{$$\frg_{\rm reg}^1\bigoplus_{i=1}^{ I}\frsp(2m_i,\R)^{g_{2i}
}\bigoplus_{j=1}^{ J}\frsp
 (2l_{r_j},\R)^{g_{3j}}$$}
 \\ \hline
    &\hspace{-1.5cm}\parbox{300pt}{$$k+\sum_ig_{2i}m_i+\sum_j 
    {g_{3j}l_{r_j}}
    =p,\; r_j\equiv 1,2,3 \,(8),\; r_j \geq 5$$}\\\hline
\end{tabular}
\\[5pt]\caption{}\end{table}
\end{center}

An explicit description of the homomorphism is given as follows. The image of the element $(X,X_1,...,X_I,Y_1,...,Y_J)$ of $\frg_1 \bigoplus_{i=1}^{ I} \frsp(4m_i,\R)\bigoplus_{j=1}^{ J} \frso(2,r_j)$ is the matrix:
\begin{center}
$
\bpm A_1 &0 &0 &Z_1 &0 &0\\
0&A_2&0&0&Z_2&0\\
0&0&A_3&0&0&Z_3\\
Z_1^* &0  &0&B_1 &0&0\\
0&Z_2^*&0&0&B_2&0\\
0&0&Z_3^*&0&0&B_3
\epm$, where 
$\bpm A_1 &Z_1\\ Z_1^*&B_1\epm
=\rho_1(X)$,\\
\end{center}
\hspace{-2.3cm}
\makebox{
$\bpm A_2 &Z_2\\ Z_2^*&B_2\epm
=\bpm A_{21} &0 &0 &Z_{21} &0 &0\\
0&\ddots&0&0&\ddots&0\\
0&0&A_{2I}&0&0&Z_{2I}\\
Z_{21}^* &0 &0 &B_{21} &0 &0\\
0&\ddots&0&0&\ddots&0\\
0&0&Z_{2I}^*&0&0&B_{2I}
\epm$, 
$\bpm A_{2i} &Z_{2i}\\ Z_{2i}^*&B_{2i}\epm
=\bpm A_{2i1} &0 &0 &Z_{2i1} &0 &0\\
0&\ddots&0&0&\ddots&0\\
0&0&A_{2ig_{2i}}&0&0&Z_{2ig_{2i}}\\
Z_{2i1}^* &0 &0 &B_{2i1} &0 &0\\
0&\ddots&0&0&\ddots&0\\
0&0&Z_{2ig_{2i}}^*&0&0&B_{2ig_{2i}}
\epm$,}
\begin{center}
$\bpm A_{2i1} &Z_{2i1}\\ Z_{2i1}^*&B_{2i1}\epm= ...= \bpm A_{2ig_{2i}} &Z_{2ig_{2i}}\\ Z_{2ig_{2i}}^*&B_{2ig_{2i}}\epm
=X_i$,
\end{center}
\hspace{-2.3cm}
\makebox{
$\bpm A_3 &Z_3\\ Z_3^*&B_3\epm
=\bpm A_{31} &0 &0 &Z_{31} &0 &0\\
0&\ddots&0&0&\ddots&0\\
0&0&A_{3J}&0&0&Z_{3J}\\
Z_{31}^* &0 &0 &B_{31} &0 &0\\
0&\ddots&0&0&\ddots&0\\
0&0&Z_{3J}^*&0&0&B_{3J}
\epm$,
$\bpm A_{3j} &Z_{3j}\\ Z_{3j}^*&B_{3j}\epm
=\bpm A_{3j1} &0 &0 &Z_{3j1} &0 &0\\
0&\ddots&0&0&\ddots&0\\
0&0&A_{3jg_{3j}}&0&0&Z_{3jg_{3j}}\\
Z_{3j1}^* &0 &0 &B_{3j1} &0 &0\\
0&\ddots&0&0&\ddots&0\\
0&0&Z_{3jg_{3j}}^*&0&0&B_{3jg_{3j}}
\epm$,}
\begin{center}
$\bpm A_{3jt} &Z_{3jt}\\ Z_{3jt}^*&B_{3jt}\epm=\rho_{r_j}(Y_j)$ (or $\rho'_{r_j}(Y_j)$ if $r_j$ is even) for $t=1,...,g_{3j}$.
\end{center}

\begin{center}
\begin{tabular}{|c|}  \hline
     $\rho:\frg\to \frso(2,p)$ \\ 
             \hline
\end{tabular}   \end{center}
If $\rho:\frg\to \frso(2,p)$ is a tight homomorphism, then $\rho$ is, up to conjugation, a composition of some of the arrows in the following diagram, where the inclusions $\frso(2,m)\to \frso(2,n)$ are inclusion as lower block, the arrows marked as $\sim$ are canonical isomorphisms and $\rho_3$ is the irreducible representation. The arrows in red mark subdiagrams that do not commute.

  
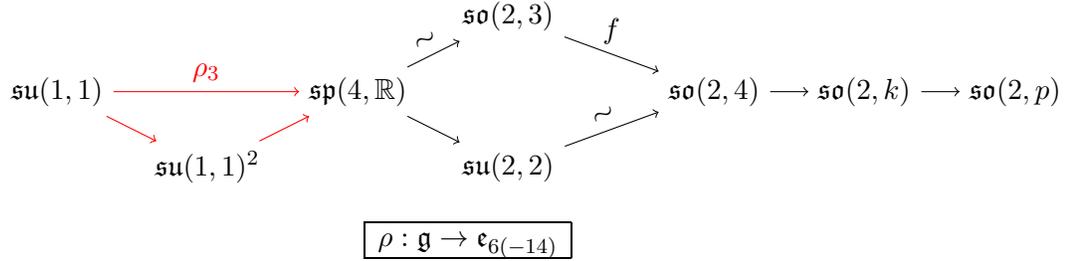
\begin{center}
 \begin{tikzpicture}[>=to]
  \matrix(M)[row sep={1cm,between origins}, column sep={2cm,between origins}]{
  &&&\node(0){$\frso(2,3)$};  \\
   \node(1){$\frsu(1,1)$};&&\node(2){$\frsp(4,\R)$};&&\node(3){$\frso(2,4)$};&\node(4){$\frso(2,k)$};&\node(5){$\frso(2,p)$};\\
   &\node(6){$\frsu(1,1)^2$};&&\node(7){$\frsu(2,2)$};\\
   };
 \draw[->] (0) -- node[above]{$f$} (3);
  \draw[->,red] (1) --  node[above]{$\rho_3$}(2); 
  \draw[->](2) -- (7); \draw[->](7) -- node[above,rotate=30]{$\sim$} (3); \draw[->](3)-- (4); \draw[->](4) -- (5);
  \draw[->,red](1)--(6);
  \draw[->,red] (6)--(2);
  \draw[->](2)--node[above,rotate=30]{$\sim$}(0);
\end{tikzpicture}
\end{center}
\begin{center}
\begin{tabular}{|c|}  \hline
     $\rho:\frg\to \fre_{6(-14)}$ \\ 
             \hline
\end{tabular}   \end{center}
Let $\rho:\frg\to \fre_{6(-14)}$ be a tight homomorphism. Then $\frg$ appears in the following diagram and $\rho$ is composition of some arrows in the following diagram. The diagram in red, and each other subdiagram containing the two arrows from $\frsu(1,1)$ doesn't commute, everything else commutes.

  
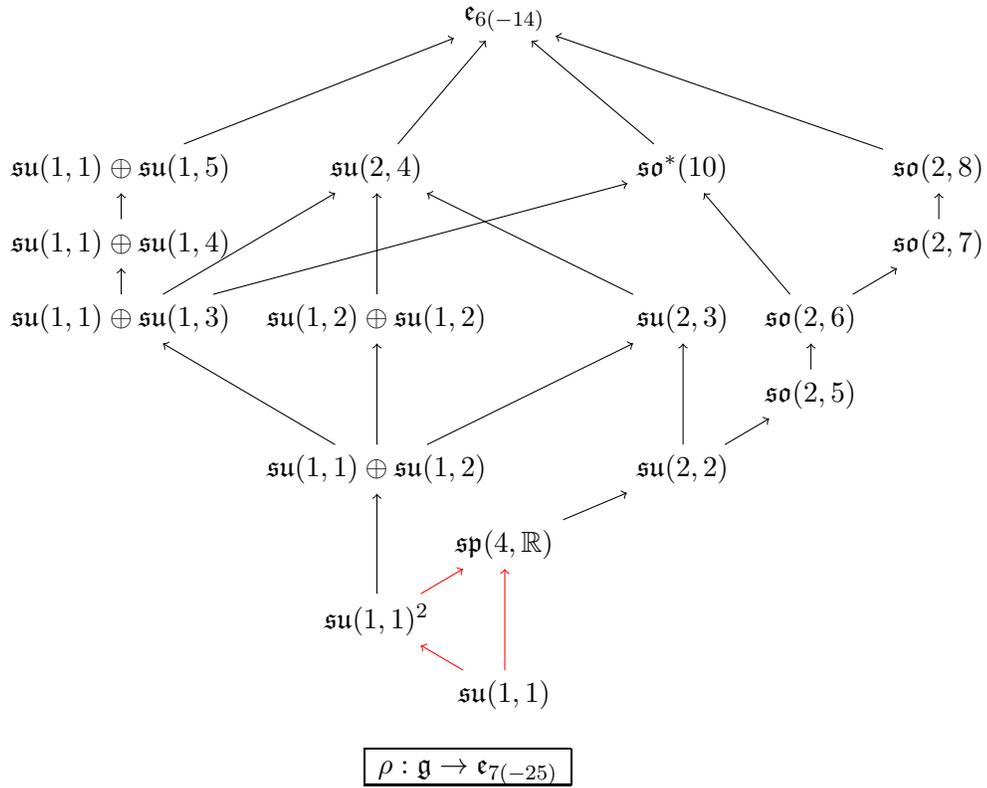
\begin{center}
 \begin{tikzpicture}[>=to]
  \matrix(M)[row sep={1cm,between origins}, column sep={1.7cm,between origins}]{
  &&&\node(00){$\fre_{6(-14)}$};\\\\
  \node(12){$\frsu(1,1)\oplus\frsu(1,5)$};&&\node(13){$\frsu(2,4)$};&&\node(14){$\frso^*(10)$};&&\node(15){$\frso(2,8)$};\\
  \node(22){$\frsu(1,1)\oplus\frsu(1,4)$};&&&&&&\node(25){$\frso(2,7)$};\\
  \node(32){$\frsu(1,1)\oplus\frsu(1,3)$};&&\node(11){$\frsu(1,2)\oplus\frsu(1,2)$};&&\node(33){$\frsu(2,3)$};&\node(34){$\frso(2,6)$};\\
  &&&&&\node(44){$\frso(2,5)$};\\
  &&\node(52){$\frsu(1,1)\oplus\frsu(1,2)$};&&\node(53){$\frsu(2,2)$};\\
  &&&\node(63){$\frsp(4,\R)$};\\
  &&\node(72){$\frsu(1,1)^2$};\\
  &&&\node(83){$\frsu(1,1)$};\\
   };
\draw[->,red] (83)--(63);\draw[->,red] (83)--(72);
\draw[->] (72)--(52);\draw[->,red] (72)--(63);
\draw[->] (63)--(53);
\draw[->] (52)--(11);\draw[->] (52)--(32);\draw[->] (52)--(33);
\draw[->] (53)--(33);\draw[->] (53)--(44);
\draw[->] (44)--(34);
\draw[->] (32)--(22);\draw[->] (32)--(13);\draw[->] (32)--(14);\draw[->] (33)--(13);
\draw[->] (34)--(14);\draw[->] (34)--(25);
\draw[->] (22)--(12);\draw[->] (25)--(15);
\draw[->] (11)--(13);
\draw[->](12)--(00);
\draw[->](13)--(00);
\draw[->](14)--(00);
\draw[->](15)--(00);
\end{tikzpicture}
\end{center}
\begin{center}
\begin{tabular}{|c|}  \hline
     $\rho:\frg\to \fre_{7(-25)}$ \\ 
             \hline
\end{tabular}   \end{center}
Unfortunately, if $\rho:\frg\to \fre_{7(-25)}$ is a tight homomorphism, we are unable to prove that $\rho$ is holomorphic, and hence we cannot exclude exotic tight embeddings, however we describe here all possible holomorphic homomorphisms as a partial step towards the full classification.
If $\rho:\frg\to \fre_{7(-25)}$ is a tight holomorphic homomorphism, then either $\frg$ is $\frso(2,p)\oplus\frsu(1,1)$ for $p$ equal to $5,7,8,9$, and the inclusion of $\frg$ in   $\fre_{7(-25)}$ is obtained by composition of the standard inclusion of $\frso(2,p)$ in $\frso(2,10)$ and the inclusion of the latter algebra as a regular subalgebra, or $\frg$ appears in the following diagram, in which, however, not all possible inclusion between classical subalgebras are displayed for readability reasons.
 
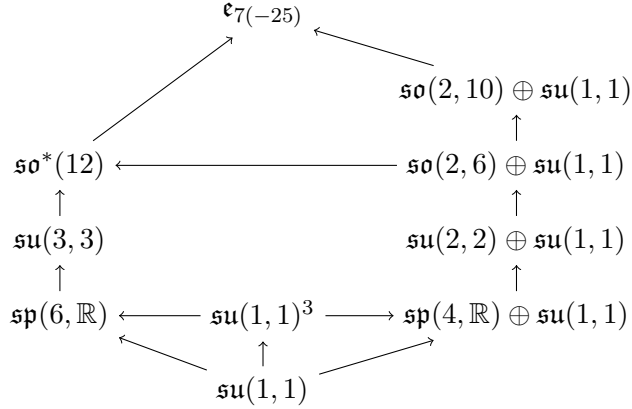
\begin{center}
  \begin{tikzpicture}
   \matrix(M)[row sep={1cm,between origins}, column sep={2.7cm,between origins}]
   {&\node(0){$\fre_{7(-25)}$};\\&&\node(1){$\frso(2,10)\oplus\frsu(1,1)$};\\
   \node(2){$\frso^*(12)$};&&\node(3){$\frso(2,6)\oplus\frsu(1,1)$};\\
   \node(4){$\frsu(3,3)$};&&\node(5){$\frsu(2,2)\oplus\frsu(1,1)$};\\
   \node(6){$\frsp(6,\R)$};&\node(7){$\frsu(1,1)^3$};&\node(8){$\frsp(4,\R)\oplus\frsu(1,1)$};\\
   &\node(9){$\frsu(1,1)$};\\};
   \draw[->] (9)--(8);   \draw[->] (9)--(7);   \draw[->] (9)--(6);   \draw[->] (7)--(6);   \draw[->] (7)--(8);   \draw[->] (6)--(4);   \draw[->] (8)--(5);   \draw[->] (4)--(2);   \draw[->] (5)--(3);   \draw[->] (3)--(1);   \draw[->] (3)--(2);
   \draw[->] (2)-- (0);
   \draw[->] (1)--(0);
  \end{tikzpicture}

 \end{center}


\begin{thebibliography}{BGPG06}

\bibitem[BGPG06]{BGPG}
S.~B. Bradlow, O.~Garc{\'{\i}}a-Prada, and P.~B. Gothen.
\newblock Maximal surface group representations in isometry groups of classical
  {H}ermitian symmetric spaces.
\newblock {\em Geom. Dedicata}, 122:185--213, 2006.

\bibitem[BIW09]{tight}
M.~Burger, A.~Iozzi, and A.~Wienhard.
\newblock Tight homomorphisms and {H}ermitian symmetric spaces.
\newblock {\em Geom. Funct. Anal.}, 19(3):678--721, 2009.

\bibitem[BIW10]{BIW}
M.~Burger, A.~Iozzi, and A.~Wienhard.
\newblock Surface group representations with maximal {T}oledo invariant.
\newblock {\em Ann. of Math. (2)}, 172(1):517--566, 2010.

\bibitem[BM99]{BMJEMS}
M.~Burger and N.~Monod.
\newblock Bounded cohomology of lattices in higher rank {L}ie groups.
\newblock {\em J. Eur. Math. Soc. (JEMS)}, 1(2):199--235, 1999.

\bibitem[C{\O}03]{CO}
J.~L. Clerc and B.~{\O}rsted.
\newblock The {G}romov norm of the {K}aehler class and the {M}aslov index.
\newblock {\em Asian J. Math.}, 7(2):269--295, 2003.

\bibitem[DT87]{DT}
A.~Domic and D.~Toledo.
\newblock The {G}romov norm of the {K}aehler class of symmetric domains.
\newblock {\em Math. Ann.}, 276(3):425--432, 1987.

\bibitem[GW08]{GW}
O.~Guichard and A.~Wienhard.
\newblock Convex foliated projective structures and the {H}itchin component for
  {${\rm PSL}_4({\bf R})$}.
\newblock {\em Duke Math. J.}, 144(3):381--445, 2008.

\bibitem[Ham13]{Ham1}
O.~Hamlet.
\newblock Tight holomorphic maps, a classification.
\newblock {\em J. Lie Theory}, 23(3):639--654, 2013.

\bibitem[Ham14]{Ham2}
O.~Hamlet.
\newblock Tight maps and holomorphicity.
\newblock {\em Transform. Groups}, 19(4):999--1026, 2014.

\bibitem[Hel78]{Helgason}
S.~Helgason.
\newblock {\em Differential geometry, {L}ie groups, and symmetric spaces},
  volume~80 of {\em Pure and Applied Mathematics}.
\newblock Academic Press, Inc. [Harcourt Brace Jovanovich, Publishers], New
  York-London, 1978.

\bibitem[HO14]{Ham3}
O.~{Hamlet} and T.~{Okuda}.
\newblock {Tight maps, exceptional spaces}.
\newblock {\em ArXiv e-prints}, November 2014.

\bibitem[Hum72]{Humphreys}
James~E. Humphreys.
\newblock {\em Introduction to {L}ie algebras and representation theory}.
\newblock Springer-Verlag, New York-Berlin, 1972.
\newblock Graduate Texts in Mathematics, Vol. 9.

\bibitem[{Poz}14]{Poz}
M.~B. {Pozzetti}.
\newblock {Maximal representations of complex hyperbolic lattices in SU(m,n)}.
\newblock {\em ArXiv e-prints}, July 2014.

\bibitem[PS69]{PS}
I.~I. Pyateskii-Shapiro.
\newblock {\em Automorphic functions and the geometry of classical domains}.
\newblock Translated from the Russian. Mathematics and Its Applications, Vol.
  8. Gordon and Breach Science Publishers, New York-London-Paris, 1969.

\bibitem[Sat65]{Satake-class}
I.~Satake.
\newblock Holomorphic imbeddings of symmetric domains into a {S}iegel space.
\newblock {\em Amer. J. Math.}, 87:425--461, 1965.

\bibitem[Sat80]{Satake}
I.~Satake.
\newblock {\em Algebraic structures of symmetric domains}, volume~4 of {\em
  Kan\^o Memorial Lectures}.
\newblock Iwanami Shoten, Tokyo; Princeton University Press, Princeton, N.J.,
  1980.

\end{thebibliography}
\end{document}